\documentclass[12pt,oneside,reqno]{article}

\usepackage{amsfonts}
\usepackage[all]{xy}
\usepackage{geometry}
\usepackage{amsmath}
\usepackage[latin5]{inputenc}
\usepackage{amssymb}
\usepackage{ntheorem}
\usepackage{geometry}
\usepackage{setspace}
\usepackage[all]{xy}
 \usepackage{graphics} 
  \usepackage{epsfig} 
 \usepackage[colorlinks=true]{hyperref}

\newtheorem*{proposition}{Proposition}

\newtheorem*{remark}{Remark}
\newtheorem*{proof}{Proof}

\begin{document}
\title{Tulczyjew's Triplet for Lie Groups II: Dynamics}

\author{O\u{g}ul Esen and Hasan G\"{u}mral}
\maketitle
\begin{center}
Department of Mathematics, Yeditepe University

34755 Kay\i \c{s}da\u{g}\i , \.{I}stanbul, Turkey

oesen@yeditepe.edu.tr \ \ \ \ hgumral@yeditepe.edu.tr

\bigskip
\end{center}

\begin{abstract}
Taking configuration space as a Lie group, the trivialized
Euler-Lagrange and Hamilton's equations are obtained and presented
as Lagrangian submanifolds of the trivialized Tulczyjew's symplectic
space. Euler-Poincar\'{e} and Lie-Poisson equations are presented as
Lagrangian submanifolds of the reduced Tulczyjew's symplectic space.
Tulczyjew's generalized Legendre transformations for trivialized and
reduced dynamics are constructed.MSC2000; Primary: 70H03, 70H05, 22E60; Secondary: 22E70, 37K65.
\end{abstract}

\section{Introduction}

Let $\mathcal{Q}$ be the configuration space of a mechanical system.
The two wings of the Tulczyjew triplet
\begin{equation}
\xymatrix{T^{\ast }T\mathcal{Q}
\ar[dr]^{\pi_{T\mathcal{Q}}}&&TT^{\ast
}\mathcal{Q}\ar[dl]^{T\pi_{\mathcal{Q}}}
\ar[rr]^{\Omega_{T^{\ast}\mathcal{Q}}^{\flat}}
\ar[dr]_{\tau_{T^{\ast }\mathcal{Q}}}
\ar[ll]_{\alpha_{\mathcal{Q}}}&&T^{\ast }T^{\ast
}\mathcal{Q}\ar@<1ex>[dl]^{\pi _{T^{\ast}\mathcal{Q}}}
\\&T\mathcal{Q}\ar@<1ex>[ul]^{dL}
&&T^{\ast}\mathcal{Q}\ar[ur]^{-dH}} \label{CTT}
\end{equation}%
defines two different special symplectic structures for Tulczyjew's
symplectic space $TT^{\ast }\mathcal{Q}$. A Lagrangian $L$ on
$T\mathcal{Q}$ (or a Hamiltonian $H$ on $T^{\ast }\mathcal{Q}$)
generates a Lagrangian
submanifold $\mathcal{S}_{TT^{\ast }\mathcal{Q}}$ of $TT^{\ast }\mathcal{Q}$%
. Legendre transformation is, then, a transformation between
realizations of the same Lagrangian submanifold with two different
functions, whose Hessians may be degenerate, and special symplectic
structures.

In \cite{EsGu14a}, the right and left global trivializations of
Tulczyjew's
triplet (\ref{CTT}) were adapted for Lie groups%

\begin{equation}
\xymatrix{^{1}T^{\ast }TG\ar[dr]_{^{1}\pi _{G\circledS
\mathfrak{g}}}&&^{1}TT^{\ast
}G\ar[ll]_{^{1}\bar{\sigma}_{G}}\ar[rr]^{^{1}\Omega _{G\circledS
\mathfrak{g}^{\ast }}^{\flat }}\ar[dl]^{^{1}T\pi
_{G}}\ar[dr]_{^{1}\tau _{G\circledS \mathfrak{g}^{\ast
}}}&&^{1}T^{\ast }T^{\ast }G\ar[dl]^{^{1}\pi _{G\circledS
\mathfrak{g}}} \\ &G\circledS \mathfrak{g}&&G\circledS
\mathfrak{g}^{\ast }}  \label{TTT}
\end{equation}
where $\mathfrak{g}$ is Lie algebra of the group $G$,
$\mathfrak{g}^{\ast }$ is the dual of $\mathfrak{g}$, the
superscript $1$ denotes the global trivialization of the first kind
that lifts the Lie group action to iterated bundles.

In this work, we shall study the Lagrangian dynamics on the global
trivialization $TG\simeq G\circledS \mathfrak{g}$ and the
Hamiltonian dynamics on $T^{\ast }G\simeq G\circledS
\mathfrak{g}^{\ast }$. We shall present trivialized Euler-Lagrange
and trivialized Hamilton's equations as Lagrangian submanifolds of
the trivialized Tulczyjew's symplectic space $\ ^{1}TT^{\ast }G$. We
shall then obtain Legendre and inverse Legendre transformations and,
arrive at Morse families on $G\circledS \left(
\mathfrak{g}\times \mathfrak{g}^{\ast }\right) $ with fibrations over $%
\mathfrak{g}^{\ast }$ and $\mathfrak{g}$, respectively.

For right invariant Lagrangian and Hamiltonian dynamics, we shall
present Euler-Poincar\'{e} and Lie-Poisson dynamics as Lagrangian
submanifolds of
the reduced Tulczyjew's symplectic space $\mathfrak{z}_{d}=\mathcal{O}%
_{\lambda }\times \mathfrak{g}^{\ast }\times \mathfrak{g}$ of the
reduced trivialized triplet
\begin{equation}
\xymatrix{\mathcal{O}_{\lambda }\times \mathfrak{g}\times
\mathfrak{g}^{\ast }\ar[dr]_{^{1}\pi _{G\circledS
\mathfrak{g}}^{G\backslash}} &&\mathcal{O}_{\lambda }\times
\mathfrak{g}^{\ast }\times
\mathfrak{g}\ar[ll]_{^{1}\bar{\sigma}_{G}^{G\backslash }}
\ar[rr]^{^{1}\Omega _{G\circledS \mathfrak{g}}^{G\backslash }}
\ar[dl]^{^{1}T\pi _{G}^{G\backslash }}\ar[dr]_{^{1}\tau _{G\circledS
\mathfrak{g}^{\ast }}^{G\backslash }}&&\mathcal{O}_{\lambda }\times
\mathfrak{g}^{\ast }\times \mathfrak{g}\ar[dl]^{^{1}\pi _{G\circledS
\mathfrak{g}^{\ast }}^{G\backslash
}}\\&\mathfrak{g}&&\mathfrak{g}^{\ast}&&}  \label{ReTT}
\end{equation}%
which will be achieved by means of the Lagrange-Dirac derivative $\mathfrak{d%
}l:\mathfrak{g\rightarrow z}_{d}$ and the Hamilton-Dirac derivative $%
\mathfrak{d}h:\mathfrak{g\rightarrow z}_{d}$ on the reduced
Lagrangian and
Hamiltonian functions, respectively. The peculiarity of the diagram (\ref%
{ReTT}) is that the right and left wings are not special symplectic
structures in the usual classical sense. We shall replace
$\mathfrak{z}_{l}$
with $T^{\ast }\mathfrak{g}$ and $\mathfrak{z}_{h}$ with $T^{\ast }\mathfrak{%
g}^{\ast }\ $to solve this problem. It turns out that, in the
reduced
diagram (\ref{ReTT}), the Morse families \ are defined on $\mathfrak{g}%
\times \mathfrak{g}^{\ast }$.

In the next section, we shall briefly review the Legendre
transformation in the sense of Tulczyjew, and recall the definitions
of special symplectic structures and Morse families. In section $3$,
we shall derive trivialized Euler-Lagrange, trivialized Hamilton's,
Euler-Poincar\'{e} and Lie-Poisson equations. In section $4$,
geometry of the trivialized Tulczyjew's triplet in diagram
(\ref{TTT}) will be summarized. We shall represent trivialized
Euler-Lagrange and trivialized Hamilton's equations as Lagrangian
submanifolds of $\ ^{1}TT^{\ast }G$. Legendre transformations of
trivialized dynamics will be established. In section $5$, we shall
start with the
reduced Tulczyjew's triplet in diagram (\ref{ReTT}) and present Euler-Poincar%
\'{e} and Lie-Poisson dynamics as Lagrangian submanifolds of $\mathfrak{z}%
_{d}$. We shall then establish the Legendre transformations of
reduced dynamics. In the last section, we shall present the example
where $G$ is the group of diffeomorphisms on a manifold.
\newpage
\section{Tulczyjew's Construction of the Legendre Transformation}

\subsection{Special Symplectic Structures}

Let $\mathcal{P}$ be a symplectic manifold carrying an exact
symplectic two form $\Omega _{\mathcal{P}}=d\vartheta
_{\mathcal{P}}$. A special symplectic
structure is a quintuple $(\mathcal{P},\pi _{\mathcal{M}}^{\mathcal{P}},%
\mathcal{M},\vartheta _{\mathcal{P}},\chi )$ where $\pi _{\mathcal{M}}^{%
\mathcal{P}}:\mathcal{P}\rightarrow \mathcal{M}$ is a fibre bundle
and $\chi :\mathcal{P}\rightarrow T^{\ast }\mathcal{M}$ is a fiber
preserving
symplectic diffeomorphism such that $\chi ^{\ast }\theta _{T^{\ast }\mathcal{%
M}}=\vartheta _{\mathcal{P}}$ for $\theta _{T^{\ast }\mathcal{M}}$
being the canonical one-form on $T^{\ast }\mathcal{M}$. $\chi $ can
be characterized uniquely by the condition
\begin{equation*}
\left\langle \chi (p),X^{\mathcal{M}}(x)\right\rangle =\left\langle
\vartheta _{\mathcal{P}}(p),X^{\mathcal{P}}(p)\right\rangle
\end{equation*}%
for each $p\in \mathcal{P}$, $\pi _{\mathcal{M}}^{\mathcal{P}}(p)=x$
and for vector fields $X^{\mathcal{M}}$ and $X^{\mathcal{P}}$
satisfying $\left( \pi _{\mathcal{M}}^{\mathcal{P}}\right) _{\ast
}X^{\mathcal{P}}=X^{\mathcal{M}}$ \cite{LaSnTu75, SnTu73, Tu80}. A
real valued function $F$ on the base manifold $\mathcal{M}\mathbb{\
}$defines a Lagrangian submanifold
\begin{equation}
\mathcal{S}_{\mathcal{P}}=\left\{ p\in \mathcal{P}:d\left( F\circ \pi _{%
\mathcal{M}}^{\mathcal{P}}\right) (p)=\vartheta _{\mathcal{P}}\left(
p\right) \right\}  \label{LagSubSSS}
\end{equation}%
of the underlying symplectic manifold $(\mathcal{P},\Omega _{\mathcal{P}%
}=d\vartheta _{\mathcal{P}})$. The function $F$ together with a
special
symplectic structure $(\mathcal{P},\pi _{\mathcal{M}}^{\mathcal{P}},\mathcal{%
M},\vartheta _{\mathcal{P}},\chi )$ are called a generating family
for the Lagrangian submanifold $\mathcal{S}_{\mathcal{P}}$. Since
$\chi $ is a symplectic diffeomorphism, it maps
$\mathcal{S}_{\mathcal{P}}$ to the image space $im\left( dF\right) $
of the exterior derivative of $F$, which is a Lagrangian submanifold
of $T^{\ast }\mathcal{M}$.

\subsection{Morse Families}

Let $\left( \mathcal{P},\pi
_{\mathcal{M}}^{\mathcal{P}},\mathcal{M}\right) $ be a fibre bundle.
The vertical bundle $V\mathcal{P}$ over $\mathcal{P}$ is
the space of vertical vectors $U\in T\mathcal{P}$ satisfying $T\pi _{%
\mathcal{M}}^{\mathcal{P}}\left( U\right) =0$. The conormal bundle
of $VP$ is defined by
\begin{equation*}
V^{0}\mathcal{P}=\left\{ \alpha \in T^{\ast
}\mathcal{P}:\left\langle \alpha ,U\right\rangle =0,\forall U\in
V\mathcal{P}\right\} .
\end{equation*}%
Let $E$ be a real-valued function on $\mathcal{P}$, then the image
$im\left( dE\right) $ of its exterior derivative is a subspace of
$T^{\ast }\mathcal{P} $. We say that $E$ is a Morse family (or an
energy function) if
\begin{equation}
T_{z}im\left( dE\right) +T_{z}V^{0}\mathcal{P}=TT^{\ast
}\mathcal{P}, \label{MorseReq}
\end{equation}%
for all $z\in im\left( dE\right) \cap V^{0}\mathcal{P}$, \cite{Be10,
LiMa, Tu80, TuUr96, TuUr99, We77}. In local coordinates $\left(
x^{a},r^{i}\right) $ on the total space $\mathcal{P}$ induced from
the coordinates $\left( x^{a}\right) $ on $\mathcal{M}$, the
requirement in Eq.(\ref{MorseReq})
reduces to the condition that the rank of the matrix%
\begin{equation*}
\left( \frac{\partial ^{2}E}{\partial x^{a}\partial x^{b}}\text{ \ \ }\frac{%
\partial ^{2}E}{\partial x^{a}\partial r^{i}}\right)
\end{equation*}%
be maximal. A Morse family $E$ on the smooth bundle $\left(
\mathcal{P},\pi _{\mathcal{M}}^{\mathcal{P}},\mathcal{M}\right) $
generates an immersed Lagrangian submanifold
\begin{equation}
\mathcal{S}_{T^{\ast }\mathcal{M}}=\left\{ \lambda _{M}\in T^{\ast }\mathcal{%
M}:T^{\ast }\pi _{\mathcal{M}}^{\mathcal{P}}(\lambda _{\mathcal{M}%
})=dE\left( p\right) \right\}  \label{LagSub}
\end{equation}%
of $\left( T^{\ast }\mathcal{M},\Omega _{T^{\ast
}\mathcal{M}}\right) $. Note that, in the definition of
$\mathcal{S}_{T^{\ast }\mathcal{M}}$, there is an intrinsic
requirement that $\pi _{\mathcal{M}}^{\mathcal{P}}\left( p\right)
=\pi _{T^{\ast }\mathcal{M}}\left( \lambda _{\mathcal{M}}\right) $.

\subsection{The Legendre Transformation}

Let $\left( \mathcal{P},\Omega _{\mathcal{P}}=d\vartheta _{\mathcal{P}%
}\right) $ be an exact symplectic manifold, and $(\mathcal{P},\pi _{\mathcal{%
M}}^{\mathcal{P}},\mathcal{M},\vartheta _{\mathcal{P}},\chi )$ be a
special symplectic structure. A function $F$ on $\mathcal{M}$
defines a Lagrangian submanifold $\mathcal{S}_{\mathcal{P}}\subset
\mathcal{P}$ as described in Eq.(\ref{LagSubSSS}). If
$\mathcal{S}_{\mathcal{P}}=im\left( \Upsilon
\right) $ is the image of a section $\Upsilon $ of $\left( \mathcal{P},\pi _{%
\mathcal{M}}^{\mathcal{P}},\mathcal{M}\right) $ then we have $\chi
\circ
\Upsilon =dF$. Assume that $(\mathcal{P},\pi _{\mathcal{M}^{\prime }}^{%
\mathcal{P}},\mathcal{M}^{\prime },\vartheta _{\mathcal{P}}^{\prime
},\chi ^{\prime })$ is another special symplectic structure
associated to the underlying symplectic space $\left(
\mathcal{P},\Omega _{\mathcal{P}}\right) $. Then, from the diagram

\begin{equation}
\xymatrix{T^{\ast }\mathcal{M} \ar[dr]^{\pi_{\mathcal{M}}}
&&\mathcal{P} \ar[dl]^{\pi^{\mathcal{P}}_{\mathcal{M}}}
\ar[rr]^{\chi^{\prime }}
\ar[dr]^{\pi^{\mathcal{P}}_{\mathcal{M}^{\prime }}}
\ar[ll]_{\chi}&&T^{\ast }\mathcal{M}^{\prime }
\ar@<1ex>[dl]^{\pi_{\mathcal{M}^{\prime }}}
\\&\mathcal{M}
\ar@<1ex>[ul]^{dF} \ar@<1ex>[ur]^{\Upsilon} &&\mathcal{M}^{\prime}
\ar[ur]^{dF^{\prime }} \ar@<1ex>[ul]^{\Upsilon^{\prime}} }
\label{TTG}
\end{equation}%
it follows that the difference $\vartheta _{\mathcal{P}}-\vartheta _{%
\mathcal{P}}^{\prime }$ of one-forms must be closed in order to satisfy $%
\Omega _{\mathcal{P}}=d\vartheta _{\mathcal{P}}=d\vartheta _{\mathcal{P}%
}^{\prime }$. When the difference is exact, there exists a function
$\Delta $
on $P$ satisfying $d\Delta =\vartheta _{\mathcal{P}}-\vartheta _{\mathcal{P}%
}^{\prime }$. If $\mathcal{S}_{\mathcal{P}}$ is the image of a section $%
\Upsilon ^{\prime }$ of the fibration $\left( \mathcal{P},\pi _{\mathcal{M}%
^{\prime }}^{\mathcal{P}},\mathcal{M}^{\prime }\right) $, then the
function
\begin{equation}
F^{\prime }=\left( F\circ \pi _{\mathcal{M}}^{\mathcal{P}}+\Delta
\right) \circ \Upsilon ^{\prime }
\end{equation}%
generates the Lagrangian submanifold $\mathcal{S}_{\mathcal{P}}$
\cite{Tu77, TuUr96, TuUr99}. This is the Legendre transformation.
If, finding a global section $\Upsilon ^{\prime }$ of $\pi
_{\mathcal{M}^{\prime }}^{\mathcal{P}}$ satisfying $im(\Upsilon
^{\prime })=\mathcal{S}_{\mathcal{P}}$ is not possible, the Legendre
transformation is not immediate. In this case, define the Morse
family
\begin{equation}
E=F\circ \pi _{\mathcal{M}}^{\mathcal{P}}+\Delta  \label{energy}
\end{equation}%
on a smooth subbundle of $\left( \mathcal{P},\pi _{\mathcal{M}^{\prime }}^{%
\mathcal{P}},\mathcal{M}^{\prime }\right) $, where $E$ satisfies the
requirement (\ref{MorseReq}) of being a Morse family. Then, $E$
generates a
Lagrangian submanifold $\mathcal{S}_{T^{\ast }\mathcal{M}^{\prime }}$ on $%
T^{\ast }\mathcal{M}^{\prime }$ as described in Eq.(\ref{LagSub}).
The inverse of $\chi ^{\prime }$ maps $\mathcal{S}_{T^{\ast
}\mathcal{M}^{\prime
}}$ to $\mathcal{S}_{\mathcal{P}}$ bijectively, that is $\mathcal{S}_{%
\mathcal{P}}=\left( \chi ^{\prime }\right) ^{-1}\left( \mathcal{S}_{T^{\ast }%
\mathcal{M}^{\prime }}\right) $.

\subsection{The Classical Tulczyjew's Triplet}

In this section, we will choose the symplectic manifold $\left( \mathcal{P}%
,\Omega _{\mathcal{P}}=d\vartheta _{\mathcal{P}}\right) $, in the diagram (%
\ref{TTG}), to be the Tulczyjew's symplectic space $\left( TT^{\ast }%
\mathcal{Q},\Omega _{TT^{\ast }\mathcal{Q}}\right) $. Here, $\Omega
_{TT^{\ast }\mathcal{Q}}$ is the symplectic two-form with two
potential one-forms $\vartheta _{1}$ and $\vartheta _{2}$ obtained
by derivations of canonical one-form $\theta _{T^{\ast
}\mathcal{Q}}$ and the symplectic-two-form $\Omega _{T^{\ast
}\mathcal{Q}}$ on $T^{\ast }\mathcal{Q} $, respectively. The
resulting special symplectic structures
\begin{equation}
\xymatrix{T^{\ast }T\mathcal{Q}
\ar[dr]^{\pi_{T\mathcal{Q}}}&&TT^{\ast
}\mathcal{Q}\ar[dl]^{T\pi_{\mathcal{Q}}}
\ar[rr]^{\Omega_{T^{\ast}\mathcal{Q}}^{\flat}}
\ar[dr]_{\tau_{T^{\ast }\mathcal{Q}}}
\ar[ll]_{\alpha_{\mathcal{Q}}}&&T^{\ast }T^{\ast
}\mathcal{Q}\ar@<1ex>[dl]^{\pi _{T^{\ast}\mathcal{Q}}}
\\&T\mathcal{Q}\ar@<1ex>[ul]^{dL}
&&T^{\ast}\mathcal{Q}\ar[ur]^{-dH}} \label{T}
\end{equation}%
where, the musical isomorphism $\Omega _{T^{\ast
}\mathcal{Q}}^{\flat }$ is induced from $\Omega _{T^{\ast
}\mathcal{Q}},$ and $\alpha _{\mathcal{Q}}$ is a diffeomorphism
constructed as a \textit{dual} of canonical involution of
$TT\mathcal{Q}$. They satisfy
\begin{equation}
\left( \Omega _{T^{\ast }\mathcal{Q}}^{\flat }\right) ^{\ast }\theta
_{T^{\ast }T^{\ast }\mathcal{Q}}=\vartheta _{1},\text{ \ \ }\alpha _{%
\mathcal{Q}}^{\ast }\theta _{T^{\ast }T\mathcal{Q}}=\vartheta _{2},
\label{pb2}
\end{equation}%
where $\theta _{T^{\ast }T^{\ast }\mathcal{Q}}$ and $\theta _{T^{\ast }T%
\mathcal{Q}}$ canonical one-forms on the cotangent bundles $T^{\ast
}T^{\ast }\mathcal{Q}$ and $T^{\ast }T\mathcal{Q}$, respectively.

The generalized Legendre transformation of Lagrangian dynamics on
the tangent bundle $T\mathcal{Q}$ can now be constructed as
follows:\ First, present the dynamics as the Lagrangian submanifold
of the Tulczyjew's symplectic space $TT^{\ast }\mathcal{Q}$. Take
the image $im\left( dL\right) $ of exterior derivative $dL$ which is
a Lagrangian submanifold of $T^{\ast
}T\mathcal{Q}$. Map $im\left( dL\right) $ by the symplectic diffeomorphism $%
\alpha _{\mathcal{Q}}$ to a Lagrangian submanifold $\mathcal{S}_{TT^{\ast }%
\mathcal{Q}}$ of $TT^{\ast }\mathcal{Q}$. Alternatively, use the
equality
\begin{equation}
\left( T\pi _{\mathcal{Q}}\right) ^{\ast }dL=\vartheta _{2}
\label{LSD}
\end{equation}%
to obtain $\mathcal{S}_{TT^{\ast }\mathcal{Q}}$. Next, generate the
same Lagrangian submanifold $\mathcal{S}_{TT^{\ast }\mathcal{Q}}$
from the right wing (the Hamiltonian side) of the triplet (\ref{T}).
To achieve this, use a Morse family $E^{L\rightarrow H}$ defined on
the Pontryagin bundle
\begin{equation*}
P\mathcal{Q}=T\mathcal{Q}\times _{\mathcal{Q}}T^{\ast }\mathcal{Q}
\end{equation*}%
over $T^{\ast }\mathcal{Q}$. The Morse family $E^{L\rightarrow H}$
defines a
Lagrangian submanifold $\mathcal{S}_{T^{\ast }T^{\ast }\mathcal{Q}}$ of $%
T^{\ast }T^{\ast }\mathcal{Q}$. The symplectic diffeomorphism
$\Omega
_{T^{\ast }\mathcal{Q}}^{\flat }$ maps $\mathcal{S}_{T^{\ast }T^{\ast }%
\mathcal{Q}}$ to the Lagrangian submanifold $\mathcal{S}_{TT^{\ast }\mathcal{%
Q}}$ obtained by means of the Lagrangian function $L$. This
completes the construction of the Legendre transformation.

For a non-degenerate Lagrangian, the Morse family $E^{L\rightarrow H}$ on $P%
\mathcal{Q}$ can be reduced to a Hamiltonian function $H$ on $T^{\ast }%
\mathcal{Q}$. For degenerate cases, a reduction of the total space $P%
\mathcal{Q}$ to a subbundle larger than $T^{\ast }\mathcal{Q}$ is
possible depending on degeneracy level of Lagrangian function
\cite{Be10}.

The inverse Legendre transformation, that is to find a Lagrangian
formulation of a Hamiltonian system, can be done pursuing the same
understanding. The musical isomorphism $\Omega _{T^{\ast }\mathcal{Q}%
}^{\flat }$ maps the image $-im\left( dH\right) $ of exterior
derivative of
a Hamiltonian $H$ on $T^{\ast }\mathcal{Q}$ to a Lagrangian submanifold $%
\mathcal{S}_{TT^{\ast }\mathcal{Q}}^{\prime }$ of the Tulczyjew's
symplectic space $TT^{\ast }\mathcal{Q}$. $\mathcal{S}_{TT^{\ast
}\mathcal{Q}}^{\prime } $ can either be defined by the equality
\begin{equation}
-\left( \tau _{T^{\ast }\mathcal{Q}}\right) ^{\ast }dH=\vartheta
_{1}, \label{HSD}
\end{equation}%
or as the image of Hamiltonian vector field $-X_{H}$. The inverse
Legendre
transformation of the dynamics is meant to generate $\mathcal{S}_{TT^{\ast }%
\mathcal{Q}}^{\prime }$ by a generating family over the tangent bundle $T%
\mathcal{Q}$. This can be done with a Morse family $E^{H\rightarrow
L}$ on the Pontryagin
bundle $P\mathcal{Q}$ with fibration over $T\mathcal{Q}$. The Morse family $%
E^{H\rightarrow L}$ defines a Lagrangian submanifold $\mathcal{S}_{T^{\ast }T%
\mathcal{Q}}^{\prime }$ of $T^{\ast }T\mathcal{Q}$ and, the
symplectic
diffeomorphism $\alpha _{\mathcal{Q}}$ maps $\mathcal{S}_{T^{\ast }T\mathcal{%
Q}}^{\prime }$ to $\mathcal{S}_{TT^{\ast }\mathcal{Q}}^{\prime }$.

In finite dimensions, introducing the coordinates $\left( \mathbf{q,p;\dot{q}%
,\dot{p}}\right) $ on $TT^{\ast }\mathcal{Q}$ induced from Darboux'
coordinates $\left( \mathbf{q,p}\right) $ on $T^{\ast }\mathcal{Q}$,
one finds the symplectomorphisms
\begin{equation*}
\alpha _{\mathcal{Q}}\left( \mathbf{q,p;\dot{q},\dot{p}}\right)
=\left(
\mathbf{q,\dot{q};\dot{p},p}\right) \text{, \ \ }\Omega _{T^{\ast }\mathcal{Q%
}}^{\flat }\left( \mathbf{q,p;\dot{q},\dot{p}}\right) =\left( \mathbf{q,p,-%
\dot{p},\dot{q}}\right) ,
\end{equation*}%
and the potential one-forms
\begin{equation}
\vartheta _{1}=\mathbf{\dot{p}}\cdot d\mathbf{q}-\mathbf{\dot{q}}\cdot d%
\mathbf{p}\text{, \ \ }\vartheta _{2}=\mathbf{\dot{p}}\cdot d\mathbf{q+p}%
\cdot d\mathbf{\dot{q},}  \label{thets}
\end{equation}%
where the difference $\vartheta _{2}-\vartheta _{1}$ is the exact one-form $%
d\left( \mathbf{p}\cdot \mathbf{\dot{q}}\right) $ on $TT^{\ast }\mathcal{Q}$%
. The Lagrangian submanifold $S_{TT^{\ast }\mathcal{Q}}$, defined in Eq.(\ref%
{LSD}), is
\begin{equation*}
\nabla _{q}L(\mathbf{q,\dot{q}})=\mathbf{\dot{p}}\text{, \ \ }\nabla _{\dot{q%
}}L(\mathbf{q,\dot{q}})=\mathbf{p},
\end{equation*}%
which can be written as a second order Euler-Lagrange equation
$d\left( \nabla _{\dot{q}}L\right) /dt=\nabla _{q}L$. The energy
function
\begin{equation*}
E^{L\rightarrow H}\left( \mathbf{q,p,\dot{q}}\right)
=\mathbf{p}\cdot \mathbf{\dot{q}}+L\left( \mathbf{q,\dot{q}}\right)
,
\end{equation*}%
satisfies the requirements, given in Eq.(\ref{MorseReq}), of being a
Morse family on the Pontryagin bundle $T\mathcal{Q}\times T^{\ast
}\mathcal{Q}$ over the cotangent bundle $T^{\ast }\mathcal{Q}$.
Hence, $E^{L\rightarrow H}$
generates a Lagrangian submanifold $\mathcal{S}_{T^{\ast }T^{\ast }\mathcal{Q%
}}$ of $T^{\ast }T^{\ast }\mathcal{Q}$ as defined in
Eq.(\ref{LagSub}). In coordinates $\left(
\mathbf{q,p,P}_{q}\mathbf{,P}_{p}\right) $ of $T^{\ast }T^{\ast
}\mathcal{Q}$, $\mathcal{S}_{T^{\ast }T^{\ast }\mathcal{Q}}$ is
given by
\begin{equation*}
\mathbf{P}_{q}=\nabla _{q}E^{L\rightarrow H}=\nabla _{q}L\text{, \ \ }%
\mathbf{P}_{p}=\nabla _{p}E^{L\rightarrow H}=\mathbf{\dot{q}}\text{, \ \ }%
\mathbf{0}=\nabla _{\dot{q}}E^{L\rightarrow H}=\mathbf{p}-\nabla _{\dot{q}}L%
\mathbf{.}
\end{equation*}%
The inverse musical isomorphism $\Omega _{T^{\ast
}\mathcal{Q}}^{\sharp }$ maps $\mathcal{S}_{T^{\ast }T^{\ast
}\mathcal{Q}}$ to $\mathcal{S}_{TT^{\ast }\mathcal{Q}}$. When the
Lagrangian function is non-degenerate, then the
Morse family reduces to the Hamiltonian function%
\begin{equation*}
H\left( \mathbf{q,p}\right) =\mathbf{p}\cdot \mathbf{\dot{q}}\left( \mathbf{%
q,p}\right) +L\left( \mathbf{q,\dot{q}}\left( \mathbf{q,p}\right)
\right)
\end{equation*}%
on $T^{\ast }\mathcal{Q}$. In coordinates $\left( \mathbf{q,p;\dot{q},\dot{p}%
}\right) $ on $TT^{\ast }\mathcal{Q}$, the Lagrangian submanifold $%
S_{TT^{\ast }\mathcal{Q}}^{\prime }$, defined in Eq.(\ref{HSD}), is
the
Hamilton's equations%
\begin{equation}
\mathbf{\dot{q}}=\nabla
_{p}{H}(\mathbf{q,p}),\;\;\;\mathbf{\dot{p}}=-\nabla
_{q}{H}(\mathbf{q,p}){.}
\end{equation}%
The Morse family
\begin{equation*}
E^{H\rightarrow L}\left( \mathbf{q,p,\dot{q}}\right)
=-\mathbf{p}\cdot \mathbf{\dot{q}}+H\left( \mathbf{q,p}\right)
\end{equation*}%
on the Pontryagin bundle $T\mathcal{Q}\times T^{\ast }\mathcal{Q}$ over $T%
\mathcal{Q}$ defines the Lagrangian submanifold $S_{T^{\ast }T\mathcal{Q}%
}^{\prime }$ of $T^{\ast }T\mathcal{Q}$. In coordinates $\left( \mathbf{q,%
\dot{q},P}_{q}\mathbf{,P}_{\dot{q}}\right) $ of $T^{\ast }T\mathcal{Q}$, $%
\mathcal{S}_{T^{\ast }T\mathcal{Q}}^{\prime }$ is given by
\begin{equation*}
\mathbf{P}_{q}=\nabla _{q}E^{H\rightarrow L}=\nabla _{q}H,\text{ \ \ }%
\mathbf{P}_{\dot{q}}=\nabla _{\dot{q}}E^{H\rightarrow
L}=\mathbf{-p},\text{ \ \ }\mathbf{0}=\nabla _{p}E^{H\rightarrow
L}=\nabla _{p}H-\mathbf{\dot{q}.}
\end{equation*}%
The inverse $\alpha _{\mathcal{Q}}^{-1}$ of the isomorphism $\alpha _{%
\mathcal{Q}}$ maps $S_{T^{\ast }T\mathcal{Q}}^{\prime }$ to $S_{TT^{\ast }%
\mathcal{Q}}^{\prime }$. When the Hamiltonian function is
non-degenerate, then the Morse family reduces to the non-degenerate
Lagrangian
\begin{equation*}
L\left( \mathbf{q,\dot{q}}\right) =-\mathbf{p}\left( \mathbf{q,\dot{q}}%
\right) \cdot \mathbf{\dot{q}}+H\left( \mathbf{q,p}\left( \mathbf{q,\dot{q}}%
\right) \right)
\end{equation*}%
on $T\mathcal{Q}$.
\newpage

\section{Dynamics on Lie Groups}
\subsection{Notations}

$G$ is a Lie group. Its Lie algebra $\mathfrak{g}\simeq T_{e}G$ is
assumed to be reflexive. The dual of $\mathfrak{g}$ is
$\mathfrak{g}^{\ast }=Lie^{\ast }\left( G\right) \simeq T_{e}^{\ast
}G$. Throughout the work, we
shall designate%
\begin{equation}
g,h\in G,\text{ \ \ }\xi ,\eta ,\zeta \in \mathfrak{g},\text{ \ \
}\mu ,\nu ,\lambda \in \mathfrak{g}^{\ast }.  \label{G}
\end{equation}%
For a tensor field which is either right or left invariant, we shall use $%
V_{g}\in T_{g}G$ or $\alpha _{g}\in T_{g}^{\ast }G$ etc... For an
arbitrary manifold $\mathcal{M}$, we shall use
\begin{equation}
u,v\in \mathcal{M}\text{, \ \ }V_{u},U_{u}\in T_{u}\mathcal{M}\text{, \ \ }%
\alpha _{u},\beta _{u},\gamma _{u}\in T_{u}^{\ast }\mathcal{M}
\label{M}
\end{equation}%
to denote vectors and one-forms over specific points. We shall
denote left and right multiplications on $G$ by $L_{g}$ and $R_{g}$,
respectively. The right inner automorphism
\begin{equation}
I_{g}=L_{g^{-1}}\circ R_{g}  \label{InnerR}
\end{equation}%
will be a right action of $G$ on $G$ satisfying $I_{g}\circ
I_{h}=I_{hg}$. The \textit{right} adjoint action $Ad_{g}=T_{e}I_{g}$
of $G$ on $\mathfrak{g} $ is defined as the tangent map of $I_{g}$
at the identity $e\in G$. The
infinitesimal \textit{right} adjoint representation $ad_{\xi }\eta $ is $%
\left[ \xi ,\eta \right] _{\mathfrak{g}}$ and it is defined as the
derivative of $Ad_{g}$ at the identity. A right invariant vector field $%
X_{\xi }^{G}$ on $G$ can be obtained by right translation
\begin{equation}
X_{\xi }^{G}\left( g\right) =T_{e}R_{g}\xi  \label{riG}
\end{equation}%
of $\xi \in \mathfrak{g}$ for each $g\in G$. The identity
\begin{equation}
\left[ \xi ,\eta \right] =\left[ X_{\xi }^{G},X_{\eta }^{G}\right]
_{JL} \label{rivf}
\end{equation}%
gives the isomorphism between $\mathfrak{g}$ and the space $\mathfrak{X}%
^{R}\left( G\right) $ of right invariant vector fields endowed with
the Jacobi-Lie bracket. The coadjoint action $Ad_{g}^{\ast }$ of $G$
on the dual $\mathfrak{g}^{\ast }$ of the Lie algebra $\mathfrak{g}$
is a right representation and is the linear algebraic dual of
$Ad_{g^{-1}}$, namely,
\begin{equation}
\left\langle Ad_{g}^{\ast }\mu ,\xi \right\rangle =\left\langle \mu
,Ad_{g^{-1}}\xi \right\rangle  \label{dist*}
\end{equation}%
holds for all $\xi \in \mathfrak{g}$ and $\mu \in \mathfrak{g}^{\ast
}$. The
infinitesimal coadjoint action $ad_{\xi }^{\ast }$ of $\mathfrak{g}$ on $%
\mathfrak{g}^{\ast }$ is the linear algebraic dual of $ad_{\xi }$.
Note that, the infinitesimal generator of the coadjoint action
$Ad_{g}^{\ast }$ is minus the infinitesimal coadjoint action
$ad_{\xi }^{\ast }$, that is, if $g^{t}\subset G$ is a curve passing
through the identity in the direction of
$\xi \in \mathfrak{g},$ then%
\begin{equation}
\left. \frac{d}{dt}\right\vert _{t=0}Ad_{g^{t}}^{\ast }\mu =-ad_{\xi
}^{\ast }\mu .  \label{Adtoad}
\end{equation}%
The\ \textit{right} trivialization maps on $TG$ and $T^{\ast }G$ are
defined to be
\begin{eqnarray}
tr_{TG}^{R} &:&TG\rightarrow G\circledS
\mathfrak{g}:U_{g}\rightarrow \left(
g,T_{g}R_{g^{-1}}U_{g}\right) , \\
tr_{T^{\ast }G}^{R} &:&T^{\ast }G\rightarrow G\circledS
\mathfrak{g}^{\ast }:\alpha _{g}\rightarrow \left( g,T_{e}^{\ast
}R_{g}\alpha _{g}\right) .
\end{eqnarray}%
We refer to \cite{EsGu14a} for further details about the right
actions and representations.

\subsection{Lagrangian dynamics}

For a Lagrangian density $L:TG\rightarrow\mathbb{R}$, define the
unique function $\bar{L}$ on $G\circledS \mathfrak{g}$ by
\begin{equation}
\bar{L}\left( g,\xi \right) =\bar{L}\circ tr_{TG}^{R}\left(
V_{g}\right) =L\left( V_{g}\right) ,  \label{Lag}
\end{equation}%
where $\xi =T_{g}R_{g^{-1}}V_{g}$. The variation of the fiber (Lie
algebra)
variable $\xi $ can be done by the reduced variational principle \cite%
{BlKrMaRa96, CeHoMaRa98, EsGu14a, HoMaRa97, HoScSt09, MaRa99,
MaRaRa91}
\begin{equation}
\delta \xi =\dot{\eta}+\left[ \xi ,\eta \right] .  \label{rvp}
\end{equation}

\begin{proposition}
A Lagrangian density $\bar{L}$ on $G\circledS \mathfrak{g}$ defines
the trivialized Euler-Lagrange dynamics
\begin{equation}
\frac{d}{dt}\frac{\delta \bar{L}}{\delta \xi }=T_{e}^{\ast }R_{g}\frac{%
\delta \bar{L}}{\delta g}+ad_{\xi }^{\ast }\frac{\delta \bar{L}}{\delta \xi }%
.  \label{preeulerlagrange}
\end{equation}

\begin{proof}
Using reduced variational principle, one computes
\begin{eqnarray*}
&&\delta \int_{b}^{a}\bar{L}\left( g,\xi \right)
dt=\int_{b}^{a}\left( \left\langle \frac{\delta \bar{L}}{\delta
g},\delta g\right\rangle _{g}+\left\langle \frac{\delta
\bar{L}}{\delta \xi },\delta \xi
\right\rangle _{e}\right) dt \\
&=&\int_{b}^{a}\left( \left\langle \frac{\delta \bar{L}}{\delta
g},\delta
g\right\rangle _{g}+\left\langle \frac{\delta \bar{L}}{\delta \xi },\dot{\eta%
}+\left[ \xi ,\eta \right] \right\rangle _{e}\right) dt \\
&=&-\left. \left\langle \frac{\delta \bar{L}}{\delta \xi },\eta
\right\rangle _{e}\right\vert _{b}^{a}+\text{ }\int_{b}^{a}\left(
\left\langle \frac{\delta \bar{L}}{\delta g},\delta g\right\rangle
_{g}+\left\langle -\frac{d}{dt}\frac{\delta \bar{L}}{\delta \xi
}+ad_{\xi }^{\ast }\frac{\delta \bar{L}}{\delta \xi },\eta
\right\rangle _{e}\right) dt
\\
&=&-\left. \left\langle \frac{\delta \bar{L}}{\delta \xi }%
,T_{g}R_{g^{-1}}\delta g\right\rangle _{e}\right\vert
_{b}^{a}+\int_{b}^{a}\left\langle \frac{\delta \bar{L}}{\delta
g},\delta
g\right\rangle _{g}+\left\langle ad_{\xi }^{\ast }\frac{\delta \bar{L}}{%
\delta \xi }-\frac{d}{dt}\frac{\delta \bar{L}}{\delta \xi }%
,T_{g}R_{g^{-1}}\delta g\right\rangle _{e}dt \\
&=&-\left. \left\langle T_{g}^{\ast }R_{g^{-1}}\frac{\delta
\bar{L}}{\delta \xi },\delta g\right\rangle _{g}\right\vert
_{a}^{b}+\int_{b}^{a}\left\langle \frac{\delta \bar{L}}{\delta g}%
+T_{g}^{\ast }R_{g^{-1}}\left( ad_{\xi }^{\ast }\frac{\delta
\bar{L}}{\delta \xi }-\frac{d}{dt}\frac{\delta \bar{L}}{\delta \xi
}\right) ,\delta g\right\rangle _{g}dt.
\end{eqnarray*}%
and the conclusion follows if $\delta g$ vanishes at boundaries.
\end{proof}
\end{proposition}

The trivialized Euler-Lagrange equation (\ref{preeulerlagrange}) is
defined over the identity $e\in G$. Because $\delta \bar{L}/\delta
g\in T_{g}^{\ast
}G$ and the functor $T_{e}^{\ast }R_{g}$ takes this to the dual space $%
\mathfrak{g}^{\ast }=T_{e}^{\ast }G$. Eq.(\ref{preeulerlagrange})
appeared in \cite{En00} with the missing operator $T_{e}^{\ast
}R_{g}$. It also appeared in some recent works \cite{CoDi11, CoDi13}
on the higher order Lagrangian and Hamiltonian dynamics on
trivialized iterated bundles of Lie groups. See also \cite{BoMa07}.

When the Lagrangian $\bar{L}$ is independent of $g,$ that is,
$\bar{L}\left( g,\xi \right) =l\left( \xi \right) $\ and $L$ on $TG$
is right invariant,
then Eq.(\ref{preeulerlagrange}) reduces to the Euler-Poincar\'{e} equation%
\begin{equation}
ad_{\xi }^{\ast }\frac{\delta l}{\delta \xi }-\frac{d}{dt}\frac{\delta l}{%
\delta \xi }=0.  \label{EPEq}
\end{equation}

\subsection{Hamiltonian dynamics}

By pushing forward the canonical one-form $\theta _{T^{\ast }G}$ and
symplectic two-form $\Omega _{T^{\ast }G}$ on $T^{\ast }G$ with the
trivialization map $tr_{T^{\ast }G}^{R}$, $G\circledS
\mathfrak{g}^{\ast }$ can be endowed with an exact symplectic
two-form $\Omega _{G\circledS \mathfrak{g}^{\ast }}=d\theta
_{G\circledS \mathfrak{g}^{\ast }}$. If
\begin{equation*}
X_{\left( \xi ,\nu \right) }^{G\circledS \mathfrak{g}^{\ast }}\left(
g,\mu \right) =\left( T_{e}R_{g}\xi ,\nu +ad_{\xi }^{R\ast }\mu
\right) ,
\end{equation*}%
is a right invariant vector field at the point $\left( g,\mu \right)
\in
G\circledS \mathfrak{g}^{\ast }$ generated by the Lie algebra element $%
\left( \xi ,\nu \right) \in Lie\left( G\circledS \mathfrak{g}^{\ast
}\right) \simeq \mathfrak{g}\circledS \mathfrak{g}^{\ast }$ then,
the values of canonical one-forms $\theta _{G\circledS
\mathfrak{g}^{\ast }}$ and $\Omega _{G\circledS \mathfrak{g}^{\ast
}}$ on $X_{\left( \xi ,\nu \right) }^{G\circledS \mathfrak{g}^{\ast
}}\left( g,\mu \right) $ are \cite{AbMa78, EsGu14a}
\begin{eqnarray}
\left\langle \theta _{G\circledS \mathfrak{g}^{\ast }},X_{\left( \xi
,\nu \right) }^{G\circledS \mathfrak{g}^{\ast }}\right\rangle \left(
g,\mu
\right) &=&\left\langle \mu ,\xi \right\rangle  \label{OhmT*G} \\
\left\langle \Omega _{G\circledS \mathfrak{g}^{\ast }};\left(
X_{\left( \xi ,\nu \right) }^{G\circledS _{R}\mathfrak{g}^{\ast
}},X_{\left( \eta ,\lambda \right) }^{G\circledS \mathfrak{g}^{\ast
}}\right) \right\rangle \left( g,\mu \right) &=&\left\langle \nu
,\eta \right\rangle -\left\langle \lambda
,\xi \right\rangle +\left\langle \mu ,\left[ \xi ,\eta \right] _{\mathfrak{g}%
}\right\rangle  \label{Ohm2T*G}
\end{eqnarray}%
which are considered for linearizations of Hamiltonian systems in \cite%
{MaRaRa91} and, for higher order dynamics in \cite{CoDi11}. Let $H$
be a
function on $T^{\ast }G$ and define $\bar{H}:G\circledS \mathfrak{g}^{\ast }%
\mathfrak{\rightarrow
\mathbb{R}
}$ by $\bar{H}\circ tr_{T^{\ast }G}^{R}=H$, that is, for $\alpha
_{g}=T_{g}^{\ast }R_{g^{-1}}\mu $, we have $\bar{H}\left( g,\mu
\right) =H\left( \alpha _{g}\right) $ and Hamilton's equations on
$\left( G\circledS
\mathfrak{g}^{\ast },\Omega _{G\circledS \mathfrak{g}^{\ast }}\right) $ are%
\begin{equation}
i_{X_{\bar{H}}^{G\circledS \mathfrak{g}^{\ast }}}\Omega _{G\circledS
\mathfrak{g}^{\ast }}=-d\bar{H}.  \label{HamEq}
\end{equation}

\begin{proposition}
The Hamiltonian vector field $X_{\bar{H}}^{G\circledS
_{R}\mathfrak{g}^{\ast }}$, defined in Eq.(\ref{HamEq}), is
generated by the element
\begin{equation*}
(\frac{\delta \bar{H}}{\delta \mu },-T_{e}^{\ast }R_{g}\left(
\frac{\delta \bar{H}}{\delta g}\right) )
\end{equation*}%
of the Lie algebra $\mathfrak{g}\circledS \mathfrak{g}^{\ast }$ of $%
G\circledS \mathfrak{g}^{\ast }$ and components of
$X_{\bar{H}}^{G\circledS
_{R}\mathfrak{g}^{\ast }}$ are given by the trivialized Hamilton's equations%
\begin{equation}
\frac{dg}{dt}=T_{e}R_{g}\left( \frac{\delta \bar{H}}{\delta \mu
}\right) ,\ \text{\ }\ \frac{d\mu }{dt}=ad_{\frac{\delta
\bar{H}}{\delta \mu }}^{\ast }\mu -T_{e}^{\ast }R_{g}\frac{\delta
\bar{H}}{\delta g}.  \label{ULP}
\end{equation}
\end{proposition}

Note that the second term on the right hand side of the second
equation in
Eq.(\ref{ULP}) is a consequence of the semidirect product structure on $%
G\circledS \mathfrak{g}^{\ast }$. Accordingly, if we let $\bar{H}$
to be independent of the group element, that is, $\bar{H}\left(
g,\mu \right) =h\left( \mu \right) $ and $H$ on $T^{\ast }G$ is
right invariant, then the trivialized Hamilton's equations
(\ref{ULP}) reduce to
\begin{equation}
\frac{dg}{dt}=T_{e}R_{g}\left( \frac{\delta h}{\delta \mu }\right)
\text{, \ \ }\frac{d\mu }{dt}=ad_{\frac{\delta h}{\delta \mu
}}^{\ast }\mu . \label{LPA}
\end{equation}

The canonical Poisson bracket on $G\circledS \mathfrak{g}^{\ast }$
is
\begin{eqnarray*}
&&\left\{ \bar{F},\bar{K}\right\} _{G\circledS \mathfrak{g}^{\ast
}}\left(
g,\mu \right) =\Omega _{G\circledS \mathfrak{g}^{\ast }}\left( X_{\bar{F}%
}^{G\circledS \mathfrak{g}^{\ast }},X_{\bar{K}}^{G\circledS \mathfrak{g}%
^{\ast }}\right) \left( g,\mu \right) \\
&=&\Omega _{G\circledS \mathfrak{g}^{\ast }}\left( X_{\left(
\frac{\delta \bar{F}}{\delta \mu },-T^{\ast }R_{g}\frac{\delta
\bar{F}}{\delta g}\right) }^{G\circledS \mathfrak{g}^{\ast
}},X_{\left( \frac{\delta \bar{K}}{\delta \mu },-T^{\ast
}R_{g}\frac{\delta \bar{K}}{\delta g}\right) }^{G\circledS
\mathfrak{g}^{\ast }}\right) \left( g,\mu \right) \\
&=&\left\langle T_{e}^{\ast }R_{g}\frac{\delta \bar{K}}{\delta g},\frac{%
\delta \bar{F}}{\delta \mu }\right\rangle -\left\langle T_{e}^{\ast }R_{g}%
\frac{\delta \bar{F}}{\delta g},\frac{\delta \bar{K}}{\delta \mu }%
\right\rangle +\left\langle \mu ,\left[ \frac{\delta \bar{F}}{\delta \mu },%
\frac{\delta \bar{K}}{\delta \mu }\right]
_{\mathfrak{g}}\right\rangle ,
\end{eqnarray*}%
for two function(al)s $\bar{F}$ and $\bar{K}$ defined on $G\circledS
\mathfrak{g}^{\ast }$. The Poisson bracket $\left\{ \text{ },\text{ }%
\right\} _{G\circledS \mathfrak{g}^{\ast }}$ is non-degenerate. When
$\bar{F}
$ and $\bar{K}$ are independent of the group variable $g\in G$, that is, $%
\bar{F}=f\left( \mu \right) $ and $\bar{K}=k\left( \mu \right) $, we
have
the Lie-Poisson bracket%
\begin{equation}
\left\{ f,k\right\} _{\mathfrak{g}^{\ast }}\left( \mu \right)
=\left\langle \mu ,\left[ \frac{\delta f}{\delta \mu },\frac{\delta
k}{\delta \mu }\right] _{\mathfrak{g}}\right\rangle
\label{LPbracket}
\end{equation}%
on the dual space $\mathfrak{g}^{\ast }$ \cite{ArKh98, HoScSt09,
MaRa99, Kr04}. This is a manifestation of the fact that the
projection $G\circledS \mathfrak{g}^{\ast }\rightarrow
\mathfrak{g}^{\ast }$ is the momentum map
for the cotangent lifted left action of $G$ on $G\circledS \mathfrak{g}%
^{\ast }$. In this case, the dynamics is driven by the Hamiltonian
vector field $X_{h}^{\mathfrak{g}^{\ast }}$ satisfying
\begin{equation*}
\left\{ f,h\right\} _{\mathfrak{g}^{\ast }}=-\left\langle df,X_{h}^{%
\mathfrak{g}^{\ast }}\right\rangle
\end{equation*}%
for a given Hamiltonian function(al) $h$ on $\mathfrak{g}^{\ast }$.
More explicitly, the value of Hamiltonian vector field
$X_{h}^{\mathfrak{g}^{\ast }}$ at $\mu \in \mathfrak{g}^{\ast }$ is
defined by the Lie-Poisson equations
\begin{equation}
\frac{d\mu }{dt}=ad_{\frac{\delta h}{\delta \mu }}^{\ast }\mu
\text{.} \label{LP}
\end{equation}%
We shall refer to both of the equations in (\ref{LPA}) and
(\ref{LP}) as Lie-Poisson equations \cite{MaRaRa91}.
\newpage
\section{The Trivialized Dynamics}

\subsection{Trivialization of the Tulczyjew's Triplet}

The tangent and cotangent lifts of the group structure on $G$ define
group
structures on $TG$ and $T^{\ast }G$, respectively. The\ trivialization maps $%
tr_{TG}^{R}$ and $tr_{T^{\ast }G}^{R}$ on $TG$ and $T^{\ast }G$ are
defined in such a way that they are not only diffeomorphisms but
also group
isomorphisms \cite{Hi06, KoMiSl93, MaRaWe84, MaRaWe84b, MaRaRa91, Mi08, Ra80}%
. For iterated bundles, with the same understanding of
trivializations we obtained \cite{EsGu14a}
\begin{eqnarray}
tr_{T^{\ast }\left( G\circledS \mathfrak{g}\right) }^{1} &:&T^{\ast
}\left(
G\circledS \mathfrak{g}\right) \rightarrow \left( G\circledS \mathfrak{g}%
\right) \circledS \left( \mathfrak{g}^{\ast }\times
\mathfrak{g}^{\ast
}\right) =\text{ }^{1}T^{\ast }TG  \label{trT*TG} \\
&:&\left( \alpha _{g},\alpha _{\xi }\right) \rightarrow \left( g,\xi
,T_{e}^{\ast }R_{g}\left( \alpha _{g}\right) +ad_{\xi }^{\ast
}\alpha _{\xi
},\alpha _{\xi }\right) ,  \notag \\
tr_{T^{\ast }\left( G\circledS \mathfrak{g}^{\ast }\right) }^{1}
&:&T^{\ast }T^{\ast }G\rightarrow \left( G\circledS
\mathfrak{g}^{\ast }\right) \circledS \left( \mathfrak{g}^{\ast
}\times \mathfrak{g}\right) =\
^{1}T^{\ast }T^{\ast }G  \label{trT*T*G} \\
&:&\left( \alpha _{g},\alpha _{\mu }\right) \rightarrow \left( g,\mu
,T_{e}^{\ast }R_{g}\left( \alpha _{g}\right) -ad_{\alpha _{\mu
}}^{\ast }\mu
,\alpha _{\mu }\right) ,  \notag \\
tr_{T\left( G\circledS \mathfrak{g}^{\ast }\right) }^{1} &:&TT^{\ast
}G\rightarrow \left( G\circledS \mathfrak{g}^{\ast }\right)
\circledS \left( \mathfrak{g}\circledS \mathfrak{g}^{\ast }\right)
=\ ^{1}TT^{\ast }G
\label{trTT*G} \\
&:&\left( V_{g},V_{\mu }\right) \rightarrow \left( g,\mu
,TR_{g^{-1}}V_{g},V_{\mu }-ad_{TR_{g^{-1}}V_{g}}^{\ast }\mu \right)
.  \notag
\end{eqnarray}%
Although, not unique, this way of trivializing iterated bundles
enables us to perform reductions of Tulczyjew's triplet.

The symplectic two-forms on the trivialized bundles $^{1}T^{\ast }TG$, $%
^{1}T^{\ast }T^{\ast }G$ and $^{1}TT^{\ast }G$ have been constructed
based on the fact that the trivialization maps $tr_{T^{\ast }\left(
G\circledS
\mathfrak{g}\right) }^{1}$, $tr_{T^{\ast }\left( G\circledS \mathfrak{g}%
^{\ast }\right) }^{1}$ and $tr_{T\left( G\circledS
\mathfrak{g}^{\ast }\right) }^{1}$ are symplectic diffeomorphisms.
The trivialized Tulczyjew's triplet
\begin{equation}
\xymatrix{^{1}T^{\ast }TG\ar[dr]_{^{1}\pi _{G\circledS
\mathfrak{g}}}&&^{1}TT^{\ast
}G\ar[ll]_{^{1}\bar{\sigma}_{G}}\ar[rr]^{^{1}\Omega _{G\circledS
\mathfrak{g}^{\ast }}^{\flat }}\ar[dl]^{^{1}T\pi
_{G}}\ar[dr]_{^{1}\tau _{G\circledS \mathfrak{g}^{\ast
}}}&&^{1}T^{\ast }T^{\ast }G\ar[dl]^{^{1}\pi _{G\circledS
\mathfrak{g}}} \\ &G\circledS \mathfrak{g}&&G\circledS
\mathfrak{g}^{\ast }}  \label{TrTT}
\end{equation}
consists of trivialized symplectic diffeomorphisms
$^{1}\bar{\sigma}_{G}$ and $^{1}\Omega _{G\circledS
\mathfrak{g}^{\ast }}^{\flat }$, and projections whose local
expressions are
\begin{eqnarray}
^{1}\bar{\sigma}_{G} &:&\text{ }^{1}TT^{\ast }G\rightarrow \text{ }%
^{1}T^{\ast }TG:\left( g,\mu ,\xi ,\nu \right) \rightarrow \left(
g,\xi ,\nu
+ad_{\xi }^{\ast }\mu ,\mu \right) ,  \label{sig1} \\
\text{ }\ ^{1}\Omega _{G\circledS \mathfrak{g}^{\ast }}^{\flat } &:&\text{ }%
^{1}TT^{\ast }G\rightarrow \text{ }^{1}T^{\ast }T^{\ast }G:\left(
g,\mu ,\xi ,\nu \right) \rightarrow \left( g,\mu ,\nu +ad_{\xi
}^{\ast }\mu ,-\xi
\right) ,  \label{ohm1} \\
\ ^{1}T\pi _{G} &:&\ ^{1}TT^{\ast }G\rightarrow G\circledS \mathfrak{g}%
:\left( g,\mu ,\xi ,\nu \right) \rightarrow \left( g,\xi \right) ,
\label{1T} \\
\text{ }^{1}\pi _{G\circledS \mathfrak{g}^{\ast }} &:&\ ^{1}T^{\ast
}T^{\ast }G\rightarrow G\circledS \mathfrak{g}^{\ast }:\left( g,\mu
,\nu ,\xi \right)
\rightarrow \left( g,\mu \right) , \\
\text{ }^{1}\pi _{G\circledS \mathfrak{g}} &:&\ ^{1}T^{\ast
}TG\rightarrow G\circledS \mathfrak{g}:\left( g,\xi ,\mu ,\nu
\right) \rightarrow \left( g,\xi \right) .
\end{eqnarray}

\subsection{Trivialized Tulczyjew's Symplectic Space}

Lie algebra of the group$\ ^{1}TT^{\ast }G$ is the semi-direct product $%
\left( \mathfrak{g}\circledS \mathfrak{g}^{\ast }\right) \circledS
\left(
\mathfrak{g}\circledS \mathfrak{g}^{\ast }\right) $. A Lie algebra element%
\begin{equation*}
\left( \xi _{2},\nu _{2},\xi _{3},\nu _{3}\right) \in \left( \mathfrak{g}%
\circledS \mathfrak{g}^{\ast }\right) \circledS \left(
\mathfrak{g}\circledS \mathfrak{g}^{\ast }\right)
\end{equation*}%
defines a right invariant vector field on $\ ^{1}TT^{\ast }G$ by the
tangent lift of right translation in $\ ^{1}TT^{\ast }G$. At a point
$\left( g,\mu
,\xi ,\nu \right) $, a right invariant vector field takes the value%
\begin{equation}
X_{\left( \xi _{2},\nu _{2},\xi _{3},\nu _{3}\right) }^{\
^{1}TT^{\ast }G}=\left( TR_{g}\xi _{2},\nu _{2}+ad_{\xi _{2}}^{\ast
}\mu ,\xi _{3}+\left[ \xi ,\xi _{2}\right] _{\mathfrak{g}},\nu
_{3}+ad_{\xi _{2}}^{\ast }\nu -ad_{\xi }^{\ast }\nu _{2}\right) .
\label{RITT*G}
\end{equation}%
By requiring the trivialization $tr_{TT^{\ast }G}^{1}$ be a
symplectic mapping, we obtain an exact symplectic structure $\Omega
_{^{1}TT^{\ast }G}$ with two potential one-forms $\theta _{1}$ and
$\theta _{2}$ on $\ ^{1}TT^{\ast }G$. At a point $\left( g,\mu ,\xi
,\nu \right) \in \
^{1}TT^{\ast }G$, the values of the potential one-forms $\theta _{1}$ and $%
\theta _{2}$ on the right invariant vector field of the form of Eq.(\ref%
{RITT*G}) are%
\begin{eqnarray}
\left\langle \theta _{1},X_{\left( \xi _{2},\nu _{2},\xi _{3},\nu
_{3}\right) }^{\ ^{1}TT^{\ast }G}\right\rangle &=&\left\langle \nu
,\xi _{2}\right\rangle -\left\langle \nu _{2},\xi \right\rangle
+\left\langle \mu
,\left[ \xi ,\eta \right] _{\mathfrak{g}}\right\rangle ,  \label{1} \\
\left\langle \theta _{2},X_{\left( \xi _{2},\nu _{2},\xi _{3},\nu
_{3}\right) }^{\ ^{1}TT^{\ast }G}\right\rangle &=&\left\langle \mu
,\xi _{3}\right\rangle +\left\langle \nu ,\xi _{2}\right\rangle
+\left\langle \mu ,\left[ \xi ,\xi _{2}\right]
_{\mathfrak{g}}\right\rangle ,  \label{2}
\end{eqnarray}%
respectively. At the same point, the value of symplectic two-form
$\Omega _{\ ^{1}TT^{\ast }G}$ on two right invariant vector fields
is
\begin{eqnarray*}
&&\left\langle \Omega _{\ ^{1}TT^{\ast }G};\left( X_{\left( \xi
_{2},\nu
_{2},\xi _{3},\nu _{3}\right) }^{\ ^{1}TT^{\ast }G},X_{\left( \bar{\xi}_{2},%
\bar{\nu}_{2},\bar{\xi}_{3},\bar{\nu}_{3}\right) }^{\ ^{1}TT^{\ast
}G}\right) \right\rangle =\left\langle \nu
_{3},\bar{\xi}_{2}\right\rangle
+\left\langle \nu _{2},\bar{\xi}_{3}\right\rangle -\left\langle \bar{\nu}%
_{2},\xi _{3}\right\rangle -\left\langle \bar{\nu}_{3},\xi
_{2}\right\rangle
\\
&&+\left\langle \nu ,\left[ \xi _{2},\bar{\xi}_{2}\right] _{\mathfrak{g}%
}\right\rangle +\left\langle \mu ,\left[ \xi _{3},\bar{\xi}_{2}\right] _{%
\mathfrak{g}}+\left[ \xi _{2},\bar{\xi}_{3}\right]
_{\mathfrak{g}}+\left[ \xi ,\left[ \xi _{2},\bar{\xi}_{2}\right]
_{\mathfrak{g}}\right] \right\rangle .
\end{eqnarray*}%
Existence of potential one-forms in Eqs.(\ref{1}) and (\ref{2})
leads us to define two special symplectic structures
\begin{eqnarray}
&&\left( \ ^{1}TT^{\ast }G,\ ^{1}\tau _{G\circledS \mathfrak{g}^{\ast }}%
\text{,}\ ^{1}T^{\ast }T^{\ast }G,\ \theta _{1},\ ^{1}\Omega
_{G\circledS
\mathfrak{g}^{\ast }}^{\flat }\right)  \label{SS1} \\
&&\left( \ ^{1}TT^{\ast }G,\ ^{1}T\pi _{G}\text{,}\ ^{1}T^{\ast
}TG,\ \theta _{2},\ ^{1}\bar{\sigma}_{G}\right) ,  \label{SS2}
\end{eqnarray}%
on the trivialized Tulczyjew's symplectic manifold $\left(
^{1}TT^{\ast }G,\Omega _{\ ^{1}TT^{\ast }G}\right) $. The structures
in Eqs.(\ref{SS1}) and (\ref{SS2}) are the right and left wings of
the trivialized Tulczyjew's triplet (\ref{TrTT}), respectively. We
refer to \cite{EsGu14a} for details.

\subsection{Trivialized Lagrangian Dynamics as a Lagrangian Submanifold}

\begin{proposition}
Let $\bar{L}$ be a Lagrangian on $G\circledS \mathfrak{g}$, then the
Lagrangian submanifold $S_{\ ^{1}TT^{\ast }G}$ defined by the
equation
\begin{equation}
\left( \ ^{1}T\pi _{G}\right) ^{\ast }d\bar{L}=\theta _{2},
\label{UnEPLag}
\end{equation}%
gives the trivialized Euler-Lagrange equations
(\ref{preeulerlagrange}). Here, the projection$\ ^{1}T\pi _{G}$ is
given by Eq.(\ref{1T}) and $\theta _{2}$ is the one-form in
Eq.(\ref{2}).
\end{proposition}

\begin{proof}
Under the global trivialization $\ ^{1}TT^{\ast }G$ of $T\left(
G\circledS \mathfrak{g}^{\ast }\right) $, given in
Eq.(\ref{trTT*G}), the Lagrangian submanifold described by
Eq.(\ref{UnEPLag}) becomes
\begin{equation}
S_{\ ^{1}TT^{\ast }G}=\left\{ \left( g,\frac{\delta \bar{L}}{\delta \xi }%
,\xi ,T^{\ast }R_{g}\frac{\delta \bar{L}}{\delta g}\right) \in \
^{1}TT^{\ast }G:\left( g,\xi \right) \in G\circledS
\mathfrak{g}\right\} . \label{Lag1}
\end{equation}%
To relate this to the trivialized Euler-Lagrange equations (\ref%
{preeulerlagrange}), we recall the reconstruction mapping%
\begin{equation}
\left( tr_{T\left( G\circledS \mathfrak{g}^{\ast }\right)
}^{1}\right) ^{-1}:\ ^{1}TT^{\ast }G\rightarrow T\left( G\circledS
\mathfrak{g}^{\ast }\right) :\left( g,\mu ,\xi ,\nu \right)
\rightarrow \left( g,\mu ,TR_{g}\xi ,\nu +ad_{\xi }^{\ast }\mu
\right) ,  \label{recons}
\end{equation}%
computed from Eq.(\ref{trTT*G}). $\left( tr_{T\left( G\circledS \mathfrak{g}%
^{\ast }\right) }^{1}\right) ^{-1}$ maps $S_{\ ^{1}TT^{\ast }G}$ to
the Lagrangian submanifold
\begin{equation}
S_{T\left( G\circledS \mathfrak{g}^{\ast }\right) }=\left\{ \left( g,\frac{%
\delta \bar{L}}{\delta \xi };TR_{g}\xi ,T^{\ast }R_{g}\frac{\delta \bar{L}}{%
\delta g}+ad_{\xi }^{\ast }\frac{\delta \bar{L}}{\delta \xi }\right)
\in
S_{\ ^{1}TT^{\ast }G}:\left( g,\xi \right) \in G\circledS \mathfrak{g}%
\right\}  \label{STT*G}
\end{equation}%
of $T\left( G\circledS \mathfrak{g}^{\ast }\right) $ and this
determines the trivialized Euler-Lagrange equations
(\ref{preeulerlagrange}).
\end{proof}

As mentioned in the context of general theory, an alternative way to obtain $%
S_{\ ^{1}TT^{\ast }G}$ is to consider a function $\bar{L}$ on
$G\circledS \mathfrak{g}$ together with the special symplectic
structure (\ref{SS2}).
This time using the trivialization $tr_{T^{\ast }\left( G\circledS \mathfrak{%
g}\right) }^{1}$ in Eq.(\ref{trT*TG}), we obtain the trivialization
of
exterior derivative $^{1}d\bar{L}:=tr_{T^{\ast }\left( G\circledS \mathfrak{g%
}\right) }^{R}\circ d\bar{L}$ which defines, through $im\left( ^{1}d\bar{L}%
\right) $, the Lagrangian submanifold
\begin{equation}
S_{\ ^{1}T^{\ast }TG}=\left\{ \left( g,\xi ,T_{e}^{\ast
}R_{g}\frac{\delta
\bar{L}}{\delta g}+ad_{\xi }^{\ast }\frac{\delta \bar{L}}{\delta \xi },\frac{%
\delta \bar{L}}{\delta \xi }\right) \in \ ^{1}T^{\ast }TG:\left(
g,\xi \right) \in G\circledS \mathfrak{g}\right\}  \label{S1T*TG}
\end{equation}%
of $\left( \ ^{1}T^{\ast }TG,\ ^{1}\Omega _{T^{\ast }\left(
G\circledS \mathfrak{g}\right) }\right) $. The inverse
$^{1}\bar{\sigma}_{G}^{-1}$ of the diffeomorphism
$^{1}\bar{\sigma}_{G}$ in Eq.(\ref{sig1}) takes the
Lagrangian submanifold $im\left( \ ^{1}d\bar{L}\right) $ in Eq.(\ref{S1T*TG}%
) to the Lagrangian submanifold $S_{\ ^{1}TT^{\ast }G}$ in
Eq.(\ref{Lag1}).

\subsection{Trivialized Hamiltonian Dynamics as a Lagrangian Submanifold}

\begin{proposition}
The Lagrangian submanifold defined by the equation
\begin{equation}
-\left( \ ^{1}\tau _{G\circledS \mathfrak{g}^{\ast }}\right) ^{\ast }d\bar{H}%
=\theta _{1}  \label{TrHamLag}
\end{equation}%
determines the trivialized Hamilton's equations (\ref{ULP}). Here,
$^{1}\tau
_{G\circledS \mathfrak{g}^{\ast }}$ is the tangent bundle projection and $%
\theta _{1}$ is the one-form in Eq.(\ref{1}).
\end{proposition}

\begin{proof}
Under the global trivialization $^{1}TT^{\ast }G$ of $TT^{\ast }G$
given in Eq.(\ref{trTT*G}), the Lagrangian submanifold
(\ref{TrHamLag}) can be described as
\begin{equation}
S_{\ ^{1}TT^{\ast }G}^{\prime }=\left\{ \left( g,\mu ,\frac{\delta \bar{H}}{%
\delta \mu },-T^{\ast }R_{g}\frac{\delta \bar{H}}{\delta g}\right)
\in \ ^{1}TT^{\ast }G:\left( g,\mu \right) \in G\circledS
\mathfrak{g}^{\ast }\right\} .  \label{Lag2}
\end{equation}%
The reconstruction mapping $\left( tr_{T\left( G\circledS
\mathfrak{g}^{\ast }\right) }^{1}\right) ^{-1}$ in Eq.(\ref{recons})
maps $S_{\ ^{1}TT^{\ast }G}^{\prime }$ to the Lagrangian submanifold
\begin{equation}
S_{T\left( G\circledS \mathfrak{g}^{\ast }\right) }^{\prime
}=\left\{ \left( TR_{g}\left( \frac{\delta \bar{H}}{\delta \mu
}\right) ,ad_{\frac{\delta
\bar{H}}{\delta \mu }}^{\ast }\mu -TR_{g}^{\ast }\frac{\delta \bar{H}}{%
\delta g}\right) \in T\left( G\circledS \mathfrak{g}^{\ast }\right)
:\left( g,\mu \right) \in G\circledS \mathfrak{g}^{\ast }\right\}
\end{equation}%
which is the image of Hamiltonian vector field
$X_{\bar{H}}^{G\circledS \mathfrak{g}^{\ast }}$ defined in
Eq.(\ref{ULP}).
\end{proof}

Alternatively, using the trivialization of the exterior derivative
\begin{equation*}
-\ ^{1}d\bar{H}=-tr_{T^{\ast }\left( G\circledS \mathfrak{g}^{\ast
}\right) }^{R}\circ d\left( \bar{H}\right)
\end{equation*}%
we obtain the Lagrangian submanifold
\begin{equation}
S_{^{1}T^{\ast }T^{\ast }G}^{\prime }=\left\{ \left( g,\mu
,ad_{\frac{\delta
\bar{H}}{\delta \mu }}^{\ast }\mu -T_{e}^{\ast }R_{g}\frac{\delta \bar{H}}{%
\delta g},-\frac{\delta \bar{H}}{\delta \mu }\right) \in \text{
}^{1}T^{\ast }T^{\ast }G:\left( g,\mu \right) \in G\circledS
\mathfrak{g}^{\ast }\right\} \label{DH}
\end{equation}%
of $^{1}T^{\ast }T^{\ast }G$. The inverse $^{1}\Omega _{G\circledS \mathfrak{%
g}^{\ast }}^{\sharp }$ of the isomorphism $^{1}\Omega _{G\circledS \mathfrak{%
g}^{\ast }}^{\flat }$ maps $S_{^{1}T^{\ast }T^{\ast }G}^{\prime }$
to the Lagrangian submanifold $S_{\ ^{1}TT^{\ast }G}^{\prime }$.
This description of $S_{^{1}TT^{\ast }G}^{\prime }$ is the usual
form of Hamilton's equation
with respect to the symplectic two-form $^{1}\Omega _{G\circledS \mathfrak{g}%
^{\ast }}$.

\subsection{Legendre Transformation for Trivialized Dynamics}

In the previous section, the trivialized Euler-Lagrange equations (\ref%
{preeulerlagrange}) have been reformulated as the Lagrangian submanifold $S_{%
\text{ }^{1}TT^{\ast }G}$ described in Eq.(\ref{Lag1}). We are now
ready to
perform the Legendre transformation, that is to describe $S_{\text{ }%
^{1}TT^{\ast }G}$ from Hamiltonian side (bundles over $G\circledS \mathfrak{g%
}^{\ast }$) of the trivialized Tulczyjew's triplet (\ref{TrTT}).

\begin{proposition}
The Lagrangian dynamics determined by the Lagrangian submanifold $%
S_{^{1}TT^{\ast }G}$ in Eq.(\ref{Lag1}) is generated by the Morse
family
\begin{equation}
E^{\bar{L}\rightarrow \bar{H}}=\left( \bar{L}\circ \text{ }^{1}T\pi
_{G}\right) +\Delta =\bar{L}\left( g,\xi \right) -\left\langle \mu
,\xi \right\rangle  \label{ener}
\end{equation}%
defined on the (right) trivialized Pontryagin bundle
$^{1}PG=G\circledS
\left( \mathfrak{g\times g}^{\ast }\right) $ over $G\circledS \mathfrak{g}%
^{\ast }$. Here, the function $\Delta =\left\langle \mu ,\xi
\right\rangle $ is defined as to satisfy
\begin{equation*}
d\Delta =\theta _{1}-\theta _{2}=-\left\langle \mu ,\xi
_{3}\right\rangle -\left\langle \nu _{2},\xi \right\rangle .
\end{equation*}
\end{proposition}

\begin{remark}
The right trivialization of the Pontryagin bundle $PG=TG\times
_{G}T^{\ast }G $ is
\begin{eqnarray*}
tr_{PG}^{1} &:&TG\times _{G}T^{\ast }G\rightarrow G\circledS \left(
\mathfrak{g\times g}^{\ast }\right) =:\text{ }^{1}PG \\
&:&\left( V_{g},\alpha _{g}\right) \rightarrow \left(
g,T_{g}R_{g^{-1}}V_{g},T_{e}^{\ast }R_{g}\alpha _{g}\right) .
\end{eqnarray*}%
In \cite{Es14}, the details of the trivialized Pontryagin bundle
$^{1}PG$ will be presented along with the implicit trivialized
Euler-Lagrange and implicit trivialized Hamiltonian dynamics on
$^{1}PG$.
\end{remark}

\begin{remark}
The potential function $\Delta $ is the value of canonical one-form
$\theta
_{G\circledS \mathfrak{g}^{\ast }}$ on the right invariant vector field $%
X_{\left( \xi ,\nu \right) }^{G\circledS \mathfrak{g}^{\ast }}$ as
given in Eq.(\ref{OhmT*G}).
\end{remark}

\begin{proof}
The Morse family $E^{\bar{L}\rightarrow \bar{H}}$, in
Eq.(\ref{ener}),
determines a Lagrangian submanifold $S_{T^{\ast }\left( G\circledS \mathfrak{%
g}^{\ast }\right) }$ which can be described by the equations
\begin{equation*}
\alpha _{g}=\frac{\delta E^{\bar{L}\rightarrow \bar{H}}}{\delta g}=\frac{%
\delta \bar{L}}{\delta g},\ \ \alpha _{\mu }=\frac{\delta E^{\bar{L}%
\rightarrow \bar{H}}}{\delta \mu }=-\xi ,\ \ 0=\frac{\delta E^{\bar{L}%
\rightarrow \bar{H}}}{\delta \xi }=\frac{\delta \bar{L}}{\delta \xi
}-\mu
\end{equation*}%
defined on the coordinates $\left( \alpha _{g},\alpha _{\mu }\right) $ of $%
T_{\left( g,\mu \right) }^{\ast }\left( G\circledS
\mathfrak{g}^{\ast
}\right) $. The trivialization $tr_{T^{\ast }\left( G\circledS \mathfrak{g}%
^{\ast }\right) }^{1}$ maps $S_{T^{\ast }\left( G\circledS \mathfrak{g}%
^{\ast }\right) }$ to the Lagrangian submanifold
\begin{equation*}
S_{\ ^{1}T^{\ast }T^{\ast }G}=\left( g,\frac{\delta \bar{L}}{\delta \xi }%
,T^{\ast }R_{g}\frac{\delta \bar{L}}{\delta g}-ad_{\xi }^{\ast
}\frac{\delta \bar{L}}{\delta \xi },-\xi \right)
\end{equation*}%
of $\ ^{1}T^{\ast }T^{\ast }G$. The musical isomorphism $\
^{1}\Omega _{G\circledS \mathfrak{g}^{\ast }}^{\sharp }$, in turn,
maps $S_{\ ^{1}T^{\ast }T^{\ast }G}$ to the Lagrangian submanifold
$S_{\ ^{1}TT^{\ast }G}$ in Eq.(\ref{Lag1}).
\end{proof}

\begin{remark}
When we have $\bar{L}=l\left( \xi \right) $, the trivialized
Euler-Lagrange equations reduce to Euler-Poincar\'{e} equations. In
this case, the Legendre
transformation is generated by the Morse family%
\begin{equation}
E^{\bar{L}\rightarrow \bar{H}}=l\left( \xi \right) -\left\langle \mu
,\xi \right\rangle .  \label{Morse1}
\end{equation}
\end{remark}

The inverse Legendre transformation defines a Lagrangian formulation
for the trivialized Hamilton's Eq.(\ref{ULP}) which is represented
by the Lagrangian submanifold $S_{\ ^{1}TT^{\ast }G}^{\prime }$
described in Eq.(\ref{Lag2}). The following proposition shows how to
find an alternative generating family for $S_{\ ^{1}TT^{\ast
}G}^{\prime }$ that will lead to its representation on the
Lagrangian side of the triplet (\ref{TrTT}).

\begin{proposition}
The Morse family%
\begin{equation}
E^{\bar{H}\rightarrow \bar{L}}=\left( -\bar{H}\circ ^{1}T\pi
_{G}\right) -\Delta =\left\langle \mu ,\xi \right\rangle
-\bar{H}\left( g,\mu \right) \label{ener2}
\end{equation}%
defined on the trivialized Pontryagin bundle $^{1}PG=G\circledS
\left( \mathfrak{g\times g}^{\ast }\right) $ over $G\circledS
\mathfrak{g}$
determines the Lagrangian submanifold $S_{^{1}TT^{\ast }G}^{\prime }$ in Eq.(%
\ref{Lag2}).
\end{proposition}

\begin{proof}
The Lagrangian submanifold $S_{T^{\ast }\left( G\circledS \mathfrak{g}%
\right) }$ of $T^{\ast }\left( G\circledS \mathfrak{g}\right) $
defined by the Morse family (\ref{ener2}) is given by
\begin{equation*}
\alpha _{g}=\frac{\delta E^{\bar{H}\rightarrow \bar{L}}}{\delta g}=-\frac{%
\delta \bar{H}}{\delta g},\text{ \ \ }\alpha _{\xi }=\frac{\delta E^{\bar{H}%
\rightarrow \bar{L}}}{\delta \xi }=\mu ,\text{ \ \ }0=\frac{\delta E^{\bar{H}%
\rightarrow \bar{L}}}{\delta \mu }=-\frac{\delta \bar{H}}{\delta \mu
}+\xi ,
\end{equation*}%
where $\left( \alpha _{g},\alpha _{\xi }\right) $ are coordinates on $%
T_{\left( g,\xi \right) }^{\ast }\left( G\circledS
\mathfrak{g}\right) $. The trivialization $tr_{T^{\ast }\left(
G\circledS \mathfrak{g}\right) }^{1}$ in Eq.(\ref{trT*TG}) maps
$S_{T^{\ast }\left( G\circledS \mathfrak{g}\right) }$ to the
Lagrangian submanifold
\begin{equation*}
S_{\ ^{1}T^{\ast }TG}=\left( g,\xi ,-T^{\ast }R_{g}\frac{\delta \bar{H}}{%
\delta g}+ad_{\frac{\delta \bar{H}}{\delta \mu }}^{\ast }\mu ,\mu
\right)
\end{equation*}%
of $\ ^{1}T^{\ast }TG$. The inverse of the isomorphism
$^{1}\bar{\sigma}_{G}$ in Eq.(\ref{sig1}) takes $S_{\ ^{1}T^{\ast
}TG}$ to the Lagrangian submanifold $S_{^{1}TT^{\ast }G}^{\prime }$
in Eq.(\ref{Lag2}).
\end{proof}

\begin{remark}
When $\bar{H}=h\left( \mu \right) $, the resulting Morse family
\begin{equation}
E^{\bar{H}\rightarrow \bar{L}}=\left\langle \mu ,\xi \right\rangle
-h\left( \mu \right)  \label{Leg2}
\end{equation}%
generates the Lie-Poisson dynamics.
\end{remark}

\section{The Reduced Dynamics}

\subsection{Reduction of Tulczyjew's triplet}

Application of the Marsden-Weinstein reduction for the left action
of $G$ to the iterated bundles in the trivialized Tulczyjew's
triplet results in symplectic projections
\begin{eqnarray}
p_{\ ^{1}T^{\ast }TG} &:&\left( \ ^{1}T^{\ast }TG,\Omega _{\
^{1}T^{\ast }TG}\right) \rightarrow \left(
\mathfrak{z}_{l}=\mathcal{O}_{\lambda }\times \mathfrak{g}\times
\mathfrak{g}^{\ast },\Omega _{\mathfrak{z}_{l}}\right)
\label{pT*TG} \\
&:&\left( g,\xi ,\lambda ,\nu \right) \rightarrow \left(
Ad_{g^{-1}}^{\ast
}\lambda ,\xi ,\nu \right) ,  \notag \\
p_{\ ^{1}T^{\ast }T^{\ast }G} &:&\left( \ ^{1}T^{\ast }T^{\ast
}G,\Omega _{\
^{1}T^{\ast }T^{\ast }G}\right) \rightarrow \left( \mathfrak{z}_{d}=\mathcal{%
O}_{\lambda }\times \mathfrak{g}^{\ast }\times \mathfrak{g},\Omega _{%
\mathfrak{z}_{d}}\right)  \label{pT*T*G} \\
&:&\left( g,\mu ,\lambda ,\xi \right) \rightarrow \left(
Ad_{g^{-1}}^{\ast
}\lambda ,\mu ,\xi \right) ,  \notag \\
p_{\ ^{1}TT^{\ast }G} &:&\left( \ ^{1}TT^{\ast }G,\Omega _{\
^{1}TT^{\ast }G}\right) \rightarrow \left(
\mathfrak{z}_{h}=\mathcal{O}_{\lambda }\times \mathfrak{g}^{\ast
}\times \mathfrak{g,}\Omega _{\mathfrak{z}_{h}}\right)
\label{pTT*G} \\
&:&\left( g,\mu ,\xi ,\nu \right) \rightarrow \left(
Ad_{g^{-1}}^{\ast }\lambda ,\mu ,\xi \right) ,  \notag
\end{eqnarray}%
into reduced spaces, where $\mathcal{O}_{\lambda }$ is the coadjoint
orbit through $\lambda \in \mathfrak{g}^{\ast }$. In the last line,
we take the fiber coordinate $v=\lambda -ad_{\xi }^{\ast }\mu $
\cite{EsGu14a} in order
to have convenience in projected coordinates as described by Eqs.(\ref%
{reddiff1}) and (\ref{reddiff2}) below, as well as in projections in Eqs.(%
\ref{tauzl})-(\ref{pizh}).

\begin{remark}
In \cite{EsGu14a}, it is shown that, the left action of $G$ on
iterated bundles can be trivialized to act on the fiber variables
$\xi ,\lambda $ and $\nu $. That means, while performing symplectic
quotients, one should
consider, literally, the orbits $G_{\lambda }\backslash (G\times \mathfrak{g}%
\times \mathfrak{g}^{\ast })$. However, to have a more clear
notation, we
prefer to take $\mathcal{O}_{\lambda }\times \mathfrak{g}\times \mathfrak{g}%
^{\ast }$ which is, indeed, diffeomorphic to the correct reduced space \cite%
{AbCaCl13}.
\end{remark}

Following \cite{EsGu14a}, we have the reduced Tulczyjew's triplet%

\begin{equation}
\xymatrix{\mathcal{O}_{\lambda }\times \mathfrak{g}\times
\mathfrak{g}^{\ast }\ar[dr]_{^{1}\pi _{G\circledS
\mathfrak{g}}^{G\backslash}} &&\mathcal{O}_{\lambda }\times
\mathfrak{g}^{\ast }\times
\mathfrak{g}\ar[ll]_{^{1}\bar{\sigma}_{G}^{G\backslash }}
\ar[rr]^{^{1}\Omega _{G\circledS \mathfrak{g}}^{G\backslash }}
\ar[dl]^{^{1}T\pi _{G}^{G\backslash }}\ar[dr]_{^{1}\tau _{G\circledS
\mathfrak{g}^{\ast }}^{G\backslash }}&&\mathcal{O}_{\lambda }\times
\mathfrak{g}^{\ast }\times \mathfrak{g}\ar[dl]^{^{1}\pi _{G\circledS
\mathfrak{g}^{\ast }}^{G\backslash
}}\\&\mathfrak{g}&&\mathfrak{g}^{\ast}&&}  \label{RTT}
\end{equation}%
consisting of the symplectic diffeomorphisms
\begin{eqnarray}
\varkappa &:&\mathfrak{z}_{d}\rightarrow \mathfrak{z}_{l}:\left(
Ad_{g^{-1}}^{\ast }\lambda ,\mu ,\xi \right) \rightarrow \left(
Ad_{g^{-1}}^{\ast }\lambda ,\xi ,\mu \right) ,  \label{reddiff1} \\
\omega ^{\flat } &:&\mathfrak{z}_{d}\rightarrow
\mathfrak{z}_{h}:\left( Ad_{g^{-1}}^{\ast }\lambda ,\mu ,\xi \right)
\rightarrow \left( Ad_{g^{-1}}^{\ast }\lambda ,\mu ,-\xi \right)
\label{reddiff2}
\end{eqnarray}%
obtained from the trivialized symplectic diffeomorphisms $^{1}\bar{\sigma}%
_{G}$ and $^{1}\Omega _{G\circledS \mathfrak{g}^{\ast }}^{\flat }$ in Eqs.(%
\ref{sig1}) and (\ref{ohm1}) by the equations
\begin{equation}
\varkappa \circ p_{\ ^{1}TT^{\ast }G}=p_{\ ^{1}T^{\ast }TG}\circ \text{ }^{1}%
\bar{\sigma}_{G},\text{ \ \ and \ \ }\omega ^{\flat }\circ p_{\
^{1}TT^{\ast
}G}=p_{\ ^{1}T^{\ast }TG}\circ \text{ }^{1}\Omega _{G\circledS \mathfrak{g}%
^{\ast }}^{\flat }.  \label{redDiff}
\end{equation}%
The projections $\tau _{\mathfrak{z}_{l}},$ $\tau
_{\mathfrak{z}_{d}}$, $\pi
_{\mathfrak{z}_{d}}$ and $\pi _{\mathfrak{z}_{h}}$ are trivial%
\begin{eqnarray}
\tau _{\mathfrak{z}_{l}} &:&\mathfrak{z}_{l}\rightarrow
\mathfrak{g}:\left( Ad_{g^{-1}}^{\ast }\lambda ,\xi ,\mu \right)
\rightarrow \xi ,  \label{tauzl}
\\
\tau _{\mathfrak{z}_{d}} &:&\mathfrak{z}_{d}\rightarrow
\mathfrak{g}:\left(
Ad_{g^{-1}}^{\ast }\lambda ,\mu ,\xi ,\right) \rightarrow \xi , \\
\pi _{\mathfrak{z}_{d}} &:&\mathfrak{z}_{d}\rightarrow
\mathfrak{g}^{\ast }:\left( Ad_{g^{-1}}^{\ast }\lambda ,\mu ,\xi
\right) \rightarrow \mu ,
\label{pizd} \\
\pi _{\mathfrak{z}_{h}} &:&\mathfrak{z}_{h}\rightarrow
\mathfrak{g}^{\ast }:\left( Ad_{g^{-1}}^{\ast }\lambda ,\mu ,\xi
\right) \rightarrow \mu . \label{pizh}
\end{eqnarray}%
In reference \cite{AbCaCl13}, Hamiltonian dynamics on
$\mathfrak{z}_{l},$ in connection with those on $T^{\ast }TG$, was
studied in detail.

\subsection{The Reduced Tulczyjew's Symplectic Space}

In order to compute vector fields and one-forms on the reduced
Tulczyjew's symplectic space $\mathfrak{z}_{d}$, we will push the
tensor fields on $\ ^{1}TT^{\ast }G$ forward by the projection $p_{\
^{1}TT^{\ast }G}$. At the point $\left( g,\mu ,\xi ,\nu \right) $,
the tangent mapping of\ $p_{\ ^{1}TT^{\ast }G}$ is
\begin{eqnarray*}
T_{\left( g,\mu ,\xi ,\nu \right) }\left( p_{\ ^{1}TT^{\ast
}G}\right) &:&T_{\left( g,\mu ,\xi ,\nu \right) }\left( \
^{1}TT^{\ast }G\right) \rightarrow T_{\left( Ad_{g^{-1}}^{\ast
}\lambda ,\mu ,\xi \right) }\left( \mathcal{O}_{\lambda }\times
\mathfrak{g}^{\ast }\times \mathfrak{g}\right)
\\
&:&\left( V_{g},V_{\mu },V_{\xi },V_{\nu }\right) \rightarrow \left(
ad_{TR_{g^{-1}}V_{g}}^{\ast }\circ Ad_{g^{-1}}^{\ast }\lambda
,V_{\mu },V_{\xi }\right) .
\end{eqnarray*}%
Pushing a right invariant vector field $X_{\left( \eta ,\upsilon ,\zeta ,%
\tilde{\upsilon}\right) }^{\ ^{1}TT^{\ast }G}$, in the form given by Eq.(\ref%
{RITT*G}), forward by $p_{\ ^{1}TT^{\ast }G}$ we arrive at the
vector field
\begin{equation}
X_{\left( \eta ,\upsilon ,\zeta \right) }^{\mathfrak{z}_{d}}\left(
Ad_{g^{-1}}^{\ast }\lambda ,\mu ,\xi \right) =\left( ad_{\eta
}^{\ast }\circ Ad_{g^{-1}}^{\ast }\lambda ,\upsilon +ad_{\eta
}^{\ast }\mu ,\zeta +\left[ \xi ,\eta \right] \right)  \label{RedVF}
\end{equation}%
on $\mathfrak{z}_{d}$. The Jacobi-Lie bracket of two such vector
fields is
\begin{equation}
\left[ X_{\left( \eta ,\upsilon ,\zeta \right) }^{\mathfrak{z}%
_{d}},X_{\left( \bar{\eta},\bar{\upsilon},\bar{\zeta}\right) }^{\mathfrak{z}%
_{d}}\right] =X_{\left( \left[ \eta ,\bar{\eta}\right] ,ad_{\bar{\eta}%
}^{\ast }\upsilon -ad_{\eta }^{\ast }\bar{\upsilon},\left[ \eta ,\bar{\zeta}%
\right] -\left[ \bar{\eta},\zeta \right] \right)
}^{\mathfrak{z}_{d}}. \label{JLXzd}
\end{equation}

\begin{proposition}
The reduced Tulczyjew's manifold
$\mathfrak{z}_{d}=\mathcal{O}_{\lambda }\times \mathfrak{g}^{\ast
}\times \mathfrak{g}$ is an exact symplectic manifold with
symplectic two-form $\Omega _{\mathfrak{z}_{d}}$, potential
one-forms $\chi _{1}$ and $\chi _{2}$ whose values on vector fields
of the form of Eq.(\ref{RedVF}) at a point $\left( Ad_{g^{-1}}^{\ast
}\lambda ,\mu ,\xi \right) \in \mathfrak{z}_{d}$ are
\begin{eqnarray}
\left\langle \Omega _{\mathfrak{z}_{d}},\left( X_{\left( \eta
,\upsilon
,\zeta \right) }^{\mathfrak{z}_{d}},X_{\left( \bar{\eta},\bar{\upsilon},\bar{%
\zeta}\right) }^{\mathfrak{z}_{d}}\right) \right\rangle
&=&\left\langle \upsilon ,\bar{\zeta}\right\rangle -\left\langle
\bar{\upsilon},\zeta \right\rangle -\left\langle \lambda ,[\eta
,\bar{\eta}]\right\rangle ,
\label{Ohmzd} \\
\left\langle \chi _{1},X_{\left( \eta ,\upsilon ,\zeta \right) }^{\mathfrak{z%
}_{d}}\right\rangle \left( Ad_{g^{-1}}^{\ast }\lambda ,\mu ,\xi
\right) &=&\left\langle \lambda ,\eta \right\rangle -\left\langle
\upsilon ,\xi
\right\rangle ,  \label{chi1} \\
\left\langle \chi _{2},X_{\left( \eta ,\upsilon ,\zeta \right) }^{\mathfrak{z%
}_{d}}\right\rangle \left( Ad_{g^{-1}}^{\ast }\lambda ,\mu ,\xi
\right) &=&\left\langle \lambda ,\eta \right\rangle +\left\langle
\mu ,\zeta \right\rangle ,  \label{chi2}
\end{eqnarray}%
respectively.
\end{proposition}

\begin{proof}
We recall definitions of the potential one-forms $\theta _{1}$ and
$\theta _{2}$ in Eqs.(\ref{1}) and (\ref{2}). Define one-forms $\chi
_{1}$ and $\chi _{2}$ on $\mathfrak{z}_{d}$ by the equations
\begin{eqnarray*}
\left\langle \theta _{1},X_{\left( \eta ,\upsilon ,\zeta ,\tilde{\upsilon}%
\right) }^{\ ^{1}TT^{\ast }G}\right\rangle \left( g,\mu ,\xi ,\nu
\right)
&=&\left\langle \chi _{1},X_{\left( \eta ,\upsilon ,\zeta \right) }^{%
\mathfrak{z}_{d}}\right\rangle \left( Ad_{g^{-1}}^{\ast }\lambda
,\mu ,\xi
\right) , \\
\left\langle \theta _{2},X_{\left( \eta ,\upsilon ,\zeta ,\tilde{\upsilon}%
\right) }^{\ ^{1}TT^{\ast }G}\right\rangle \left( g,\mu ,\xi ,\nu
\right)
&=&\left\langle \chi _{2},X_{\left( \eta ,\upsilon ,\zeta \right) }^{%
\mathfrak{z}_{d}}\right\rangle \left( Ad_{g^{-1}}^{\ast }\lambda
,\mu ,\xi \right) .
\end{eqnarray*}%
The exterior derivative of $\chi _{1}$ in Eq.(\ref{chi1}) gives the
symplectic two-form $\Omega _{\mathfrak{z}_{d}}$. Using the
invariant definition of exterior derivative we obtain
\begin{eqnarray}
\left\langle \Omega _{\mathfrak{z}_{d}};\left( X_{\left( \eta
,\upsilon
,\zeta \right) }^{\mathfrak{z}_{d}},X_{\left( \bar{\eta},\bar{\upsilon},\bar{%
\zeta}\right) }^{\mathfrak{z}_{d}}\right) \right\rangle &=&X_{\left(
\eta ,\upsilon ,\zeta \right) }^{\mathfrak{z}_{d}}\left\langle \chi
_{1},X_{\left( \bar{\eta},\bar{\upsilon},\bar{\zeta}\right) }^{\mathfrak{z}%
_{d}}\right\rangle -X_{\left(
\bar{\eta},\bar{\upsilon},\bar{\zeta}\right)
}^{\mathfrak{z}_{d}}\left\langle \chi _{1},X_{\left( \eta ,\upsilon
,\zeta
\right) }^{\mathfrak{z}_{d}}\right\rangle  \notag \\
&&-\left\langle \chi _{1},\left[ X_{\left( \eta ,\upsilon ,\zeta \right) }^{%
\mathfrak{z}_{d}},X_{\left( \bar{\eta},\bar{\upsilon},\bar{\zeta}\right) }^{%
\mathfrak{z}_{d}}\right] \right\rangle  \notag \\
&=&-\left\langle \bar{\upsilon},\zeta +\left[ \xi ,\eta \right]
\right\rangle -\left( -\left\langle \upsilon ,\bar{\zeta}+\left[ \xi ,\bar{%
\eta}\right] \right\rangle \right)  \notag \\
&&-\left\langle \lambda ,\left[ \eta ,\bar{\eta}\right]
\right\rangle
-\left\langle ad_{\bar{\eta}}^{\ast }\upsilon -ad_{\eta }^{\ast }\bar{%
\upsilon},\xi \right\rangle  \notag \\
&=&\left\langle \upsilon ,\bar{\zeta}\right\rangle -\left\langle \bar{%
\upsilon},\zeta \right\rangle -\left\langle \lambda ,[\eta ,\bar{\eta}%
]\right\rangle ,
\end{eqnarray}%
where we used the fact that $\left\langle \lambda ,\eta
\right\rangle $ is a
constant for a fixed $\lambda $, and the Jacobi Lie bracket in Eq.(\ref%
{JLXzd}). Similarly, we can show $\Omega _{\mathfrak{z}_{d}}=d\chi
_{2}$.
\end{proof}

It follows from Eqs.(\ref{chi1}) and (\ref{chi2}) that the difference%
\begin{equation}
\chi _{2}-\chi _{1}=d\left\langle \mu ,\xi \right\rangle =d\Delta
\label{delta}
\end{equation}%
is an exact one-form on $\mathfrak{z}_{d}$.

The explicit expressions of the reduced symplectic two-forms $\Omega _{%
\mathfrak{z}_{l}}$ and $\Omega _{\mathfrak{z}_{h}}$ on the product bundles $%
\mathfrak{z}_{l}$ and $\mathfrak{z}_{h}$ can be obtained by the
pull-back of $\Omega _{\mathfrak{z}_{d}}$ in Eq.(\ref{Ohmzd}) with
the symplectic diffeomorphisms $\varkappa $ and $\omega ^{\flat }$
in Eqs.(\ref{reddiff1})
and (\ref{reddiff2}), respectively. The symplectic two-form $\Omega _{%
\mathfrak{z}_{d}}$ is an example of the reduced product dynamics
defined in proposition $5.4$ of \cite{VaKaMa09}.

\subsection{Euler-Poincar\'{e} dynamics as a Lagrangian submanifold}

When $\bar{L}=l\left( \xi \right) $, the trivialized exterior derivative $%
^{1}d\bar{L}$ in Eq.(\ref{S1T*TG}) becomes%
\begin{equation}
\ ^{1}dl:\mathfrak{g}\rightarrow \text{ }^{1}T^{\ast }TG:\xi
\rightarrow
\left( g,\xi ,ad_{\xi }^{\ast }\frac{\delta l}{\delta \xi },\frac{\delta l}{%
\delta \xi }\right) .  \label{1Dl}
\end{equation}%
The image of trivialized exterior derivative $^{1}dl$ can be reduced to $%
\mathfrak{z}_{l}$ by composition with the projection map $p_{\
^{1}T^{\ast
}TG}$ in Eq.(\ref{pT*TG}). That is, we define%
\begin{equation*}
d^{G\backslash }l=p_{\ ^{1}T^{\ast }TG}\circ \text{ }^{1}dl:\mathfrak{g}%
\rightarrow \mathfrak{z}_{l}:\xi \rightarrow \left( ad_{\xi }^{\ast }\frac{%
\delta l}{\delta \xi },\xi ,\frac{\delta l}{\delta \xi }\right) ,
\end{equation*}%
where we choose $g=e$ without loss of generality. Applying the inverse $%
\varkappa ^{-1}$ of the symplectic diffeomorphism $\varkappa :\mathfrak{z}%
_{d}\rightarrow \mathfrak{z}_{l}$ in Eq.(\ref{reddiff1}), we define
Lagrange-Dirac derivative%
\begin{equation}
\mathfrak{d}l=\varkappa ^{-1}\circ d^{G\backslash }l=\varkappa
^{-1}\circ
p_{\ ^{1}T^{\ast }TG}\circ \text{ }^{1}dl:\mathfrak{\mathfrak{g}\rightarrow z%
}_{d}:\xi \rightarrow \left( ad_{\xi }^{\ast }\frac{\delta l}{\delta \xi },%
\frac{\delta l}{\delta \xi },\xi \right) .  \label{LagDir}
\end{equation}

\begin{proposition}
The image of Lagrange-Dirac derivative $\mathfrak{d}l$, in Eq.(\ref{LagDir}%
), is a Lagrangian submanifold $s_{\mathfrak{z}_{d}}$ of $\left( \mathfrak{z}%
_{d},\Omega _{\mathfrak{z}_{d}}\right) $ defining the
Euler-Poincar\'{e} equations (\ref{EPEq}).

\begin{proof}
Since, the trivialization map $tr_{T^{\ast }T^{\ast }G}^{1}$ is
symplectic, the image of $^{1}dl$ is a Lagrangian submanifold of
$^{1}T^{\ast }TG$. The
projection $p_{\ ^{1}T^{\ast }TG}$ is symplectic, and hence the image of $%
d^{G\backslash }l$ is a Lagrangian submanifold $s_{\mathfrak{z}_{l}}$ of $%
\mathfrak{z}_{l}$. The inverse symplectic diffeomorphism $\varkappa
^{-1}$ maps this Lagrangian submanifold $s_{\mathfrak{z}_{l}}$ to a
Lagrangian
submanifold $s_{\mathfrak{z}_{d}}$ of $\mathfrak{z}_{d}$. So, the image $s_{%
\mathfrak{z}_{d}}$ of $\mathfrak{d}l$ is a Lagrangian submanifold of
$\left( \mathfrak{z}_{d},\Omega _{\mathfrak{z}_{d}}\right) $. Under
the global trivialization, $s_{\mathfrak{z}_{d}}$ is obtained to be
\begin{equation}
s_{\mathfrak{z}_{d}}=\left\{ \left( ad_{\xi }^{\ast }\frac{\delta
l}{\delta \xi },\frac{\delta l}{\delta \xi },\xi \right) \in
\mathfrak{z}_{d}:\xi \in \mathfrak{\mathfrak{g}}\right\} .
\label{szd}
\end{equation}%
When $\bar{L}=l\left( \xi \right) $, the Lagrangian submanifold
$S_{\ ^{1}TT^{\ast }G}$ in Eq.(\ref{Lag1}) reduces to
\begin{equation*}
s_{\ ^{1}TT^{\ast }G}=\left\{ \left( e,\frac{\delta l}{\delta \xi
},\xi ,0\right) \in \ ^{1}TT^{\ast }G:\xi \in
\mathfrak{\mathfrak{g}}\right\} ,
\end{equation*}%
and the first definition in Eq.(\ref{redDiff}) shows that the projection of $%
s_{\ ^{1}TT^{\ast }G}$ by $p_{\ ^{1}TT^{\ast }G}$ is
$s_{\mathfrak{z}_{d}}$. The reconstruction mapping$\ ^{1}TT^{\ast
}G\rightarrow T\left( G\circledS \mathfrak{g}^{\ast }\right) $ in
Eq.(\ref{recons}) takes $s_{\ ^{1}TT^{\ast }G}$ to the Lagrangian
submanifold
\begin{equation*}
s_{T\left( G\circledS \mathfrak{g}^{\ast }\right) }=\left\{ \left( e,\frac{%
\delta l}{\delta \xi };\xi ,ad_{\xi }^{\ast }\frac{\delta l}{\delta \xi }%
\right) \in T\left( G\circledS \mathfrak{g}^{\ast }\right) :\xi \in
\mathfrak{\mathfrak{g}}\right\}
\end{equation*}%
of $T\left( G\circledS \mathfrak{g}^{\ast }\right) $, and this
defines Euler-Poincar\'{e} equations (\ref{EPEq}).
\end{proof}
\end{proposition}

Alternatively, the formulation that uses the de Rham exterior
derivative and the potential one-form $\chi _{2}$ in Eq.(\ref{chi2})
goes as follows.

\begin{proposition}
The identity
\begin{equation*}
\tau _{\mathfrak{z}_{d}}^{\ast }dl=\chi _{2}
\end{equation*}%
defines the Lagrangian submanifold $s_{\mathfrak{z}_{d}}$ in
Eq.(\ref{szd}),
hence the Euler-Poincar\'{e} equations (\ref{EPEq}). Here, $\tau _{\mathfrak{%
z}_{d}}$ is the projection $\mathfrak{z}_{d}\rightarrow
\mathfrak{g}$, $dl$ is the (de Rham) exterior derivative of $l$ on
$\mathfrak{g}$, and $\chi _{2} $ is the potential one-form in
Eq.(\ref{chi2}).

\begin{proof}
We compute the value of exact one-from $\tau
_{\mathfrak{z}_{d}}^{\ast
}dl=d\left( l\circ \tau _{\mathfrak{z}_{d}}\right) $ on a vector field $%
X_{\left( \eta ,\upsilon ,\zeta \right) }^{\mathfrak{z}_{d}}$ in Eq.(\ref%
{RedVF}). At a point\ $\left( Ad_{g^{-1}}^{\ast }\lambda ,\mu ,\xi
\right) ,$ we have
\begin{eqnarray*}
\left\langle \tau _{\mathfrak{z}_{d}}^{\ast }dl,X_{\left( \eta
,\upsilon ,\zeta \right) }^{\mathfrak{z}_{d}}\right\rangle \left(
Ad_{g^{-1}}^{\ast
}\lambda ,\mu ,\xi \right) &=&\left\langle dl,\left( \tau _{\mathfrak{z}%
_{d}}\right) _{\ast }X_{\left( \eta ,\upsilon ,\zeta \right) }^{\mathfrak{z}%
_{d}}\right\rangle \\
&=&\left\langle \frac{\delta l}{\delta \xi },\zeta +\left[ \xi ,\eta
\right]
\right\rangle \\
&=&\left\langle \frac{\delta l}{\delta \xi },\zeta \right\rangle
+\left\langle ad_{\xi }^{\ast }\frac{\delta l}{\delta \xi },\eta
\right\rangle ,
\end{eqnarray*}%
where $\left( \tau _{\mathfrak{z}_{d}}\right) _{\ast }X_{\left( \eta
,\upsilon ,\zeta \right) }^{\mathfrak{z}_{d}}$ is the push forward
of the vector field $X_{\left( \eta ,\upsilon ,\zeta \right)
}^{\mathfrak{z}_{d}}$ by the projection $\tau _{\mathfrak{z}_{d}}$
from $\mathfrak{z}_{d}$ to its
third factor $\mathfrak{g}$, that is, to the vector $\zeta +\left[ \xi ,\eta %
\right] $ in $T_{\xi }\mathfrak{g}\simeq \mathfrak{g}$. Equating this to $%
\left\langle \chi _{2},X_{\left( \eta ,\upsilon ,\zeta \right) }^{\mathfrak{z%
}_{d}}\right\rangle $ in Eq.(\ref{chi2}) gives the Lagrangian submanifold $%
s_{\mathfrak{z}_{d}}=im\left( \mathfrak{d}l\right) $ in Eq.(\ref{szd}) via%
\begin{equation*}
\lambda =ad_{\xi }^{\ast }\frac{\delta l}{\delta \xi }\text{ \ \ and \ }\mu =%
\frac{\delta l}{\delta \xi }
\end{equation*}%
in coordinates $\left( \lambda ,\mu ,\xi \right) $ of
$\mathfrak{z}_{d}$.
\end{proof}
\end{proposition}

\subsection{Lie-Poisson dynamics as a Lagrangian submanifold}

Consider a Hamiltonian function $\bar{H}$ on $G\circledS
\mathfrak{g}^{\ast } $ and define $h:\mathfrak{g}^{\ast }\rightarrow
\mathbb{R}
$ by $\bar{H}=h\left( \mu \right) $. With the trivialized exterior derivative%
\begin{equation*}
-\ ^{1}dh:\mathfrak{g}^{\ast }\rightarrow \text{ }^{1}T^{\ast
}T^{\ast }G:\left( g,\mu \right) \rightarrow \left( g,\mu
,ad_{\frac{\delta h}{\delta \mu }}^{\ast }\mu ,-\frac{\delta
h}{\delta \mu }\right) ,
\end{equation*}%
and the projection $p_{\ ^{1}T^{\ast }T^{\ast }G}:$ $^{1}T^{\ast
}T^{\ast
}G\rightarrow \mathfrak{z}_{h}$ in Eq.(\ref{pT*T*G}), we define%
\begin{equation*}
-d^{G\backslash }h=p_{\ ^{1}T^{\ast }T^{\ast }G}\circ \ ^{1}d\left(
\mathfrak{-}h\right) :\mathfrak{g}^{\ast }\mathfrak{\rightarrow
z}_{h}:\mu
\rightarrow \left( ad_{\frac{\delta h}{\delta \mu }}^{\ast }\mu ,\mu ,-\frac{%
\delta h}{\delta \mu }\right)
\end{equation*}%
by choosing $g=e.$ Applying the inverse $\omega ^{\sharp }$ of the
symplectic diffeomorphism $\omega ^{\flat }$ in Eq.(\ref{reddiff2}),
we obtain the Hamilton-Dirac derivative
\begin{equation}
-\mathfrak{d}h=\omega ^{\sharp }\circ d^{G\backslash }\left( -h\right) :%
\mathfrak{\mathfrak{g}}^{\ast }\mathfrak{\rightarrow z}_{d}:\mu
\rightarrow
\left( ad_{\frac{\delta h}{\delta \mu }}^{\ast }\mu ,\mu ,\frac{\delta h}{%
\delta \mu }\right) .  \label{HamDir}
\end{equation}

\begin{proposition}
The image of the Hamilton-Dirac derivative $-\mathfrak{d}h$ is a
Lagrangian
submanifold $s_{\mathfrak{z}_{d}}^{\prime }$ of $\left( \mathfrak{z}%
_{d},\Omega _{\mathfrak{z}_{d}}\right) $ and it defines the
Lie-Poisson equations (\ref{LP}).

\begin{proof}
The image of $-\ ^{1}dh$ is a Lagrangian submanifold of $^{1}T^{\ast
}T^{\ast }G$ and $p_{\ ^{1}T^{\ast }T^{\ast }G}$ maps this
Lagrangian
submanifold to a Lagrangian submanifold $s_{\mathfrak{z}_{h}}^{\prime }$ of $%
\mathfrak{z}_{h}$. The musical isomorphism $\omega ^{\sharp }$ takes $s_{%
\mathfrak{z}_{h}}^{\prime }$ to the Lagrangian submanifold $s_{\mathfrak{z}%
_{d}}^{\prime }$ of $\left( \mathfrak{z}_{d},\Omega _{\mathfrak{z}%
_{d}}\right) $. Thus,
\begin{equation}
s_{\mathfrak{z}_{d}}^{\prime }=im\left( -\mathfrak{d}h\right)
=\omega ^{\sharp }\circ im\left( d^{G\backslash }\left( -h\right)
\right) =\omega
^{\sharp }\circ p_{\ ^{1}T^{\ast }T^{\ast }G}\circ im\left( -\text{ }%
^{1}dh\right) .  \label{szd2}
\end{equation}%
From Eqs. (\ref{pT*T*G}) and (\ref{pTT*G}) we have
\begin{equation*}
\omega ^{\sharp }\circ p_{\ ^{1}T^{\ast }T^{\ast }G}=p_{\
^{1}TT^{\ast }G}\circ \ ^{1}\Omega _{G\circledS \mathfrak{g}^{\ast
}}^{\sharp },
\end{equation*}%
where $^{1}\Omega _{G\circledS \mathfrak{g}^{\ast }}^{\sharp }$ is
the inverse of the isomorphism $^{1}\Omega _{G\circledS
\mathfrak{g}^{\ast
}}^{\flat }$ in Eq.(\ref{ohm1}). This implies that $s_{\mathfrak{z}%
_{d}}^{\prime }$ is the projection $p_{\ ^{1}TT^{\ast }G}\left( s_{\
^{1}TT^{\ast }G}^{\prime }\right) $ of the Lagrangian submanifold
\begin{equation*}
s_{\ ^{1}TT^{\ast }G}^{\prime }=\left\{ \left( g,\mu ;\frac{\delta
h}{\delta \mu },0\right) \in \ ^{1}TT^{\ast }G:\mu \in
\mathfrak{\mathfrak{g}}^{\ast }\right\}
\end{equation*}%
obtained from $S_{\ ^{1}TT^{\ast }G}^{\prime }$ in Eq.(\ref{Lag2})
by substituting $\bar{H}=h\left( \mu \right) $. The reconstruction
mapping$\ ^{1}TT^{\ast }G\rightarrow T\left( G\circledS
\mathfrak{g}^{\ast }\right) $ in the Eq.(\ref{recons}) takes $s_{\
^{1}TT^{\ast }G}^{\prime }$ to the Lagrangian submanifold
\begin{equation*}
s_{T\left( G\circledS \mathfrak{g}^{\ast }\right) }^{\prime
}=\left\{ \left(
g,\mu ;TR_{g}\frac{\delta h}{\delta \mu },ad_{\frac{\delta h}{\delta \mu }%
}^{\ast }\mu \right) \in T\left( G\circledS \mathfrak{g}^{\ast
}\right) :\mu \in \mathfrak{\mathfrak{g}}^{\ast }\right\}
\end{equation*}%
of $T\left( G\circledS \mathfrak{g}^{\ast }\right) $ which is the
Lie-Poisson equation (\ref{LPA}) equivalent to (\ref{LP}).
\end{proof}
\end{proposition}

Alternatively, with exterior derivative and the potential one-form
$\chi _{1} $ in Eq.(\ref{chi1}), we have

\begin{proposition}
The equation
\begin{equation*}
-\pi _{\mathfrak{z}_{d}}^{\ast }dh=\chi _{1}
\end{equation*}%
defines the Lagrangian submanifold $s_{\mathfrak{z}_{d}}^{\prime }$ in Eq.(%
\ref{szd2}) and gives the Lie-Poisson equations (\ref{LP}). Here, $\pi _{%
\mathfrak{z}_{d}}$ is the projection $\mathfrak{z}_{d}\rightarrow \mathfrak{g%
}^{\ast }$, $dh$ is the (de Rham) exterior derivative of $h$ on $\mathfrak{g}%
^{\ast }$, and $\chi _{1}$ is the potential one-form in
Eq.(\ref{chi1}).

\begin{proof}
To prove this identity, we compute the value of exact one-from $\pi _{%
\mathfrak{z}_{d}}^{\ast }dh=d\left( h\circ \pi
_{\mathfrak{z}_{d}}\right) $
on a vector field $X_{\left( \eta ,\upsilon ,\zeta \right) }^{\mathfrak{z}%
_{d}}$ over $\mathfrak{z}_{d}$. At a point\ $\left(
Ad_{g^{-1}}^{\ast
}\lambda ,\mu ,\xi \right) ,$ we have%
\begin{eqnarray*}
\left\langle -\pi _{\mathfrak{z}_{d}}^{\ast }dh,X_{\left( \eta
,\upsilon ,\zeta \right) }^{\mathfrak{z}_{d}}\right\rangle \left(
Ad_{g^{-1}}^{\ast
}\lambda ,\mu ,\xi \right)  &=&-\left\langle dh,\left( \pi _{\mathfrak{z}%
_{d}}\right) _{\ast }X_{\left( \eta ,\upsilon ,\zeta \right) }^{\mathfrak{z}%
_{d}}\right\rangle  \\
&=&-\left\langle \frac{\delta h}{\delta \mu },\upsilon +ad_{\eta
}^{\ast
}\mu \right\rangle  \\
&=&-\left\langle \frac{\delta h}{\delta \mu },\upsilon \right\rangle
+\left\langle ad_{\frac{\delta h}{\delta \mu }}^{\ast }\mu ,\eta
\right\rangle ,
\end{eqnarray*}%
where $\left( \pi _{\mathfrak{z}_{d}}\right) _{\ast }X_{\left( \eta
,\upsilon ,\zeta \right) }^{\mathfrak{z}_{d}}$ is the push forward
of the vector field $X_{\left( \eta ,\upsilon ,\zeta \right)
}^{\mathfrak{z}_{d}}$ by the projection $\pi _{\mathfrak{z}_{d}}$
from $\mathfrak{z}_{d}$ to its second factor $\mathfrak{g}^{\ast }$,
that is to the dual vector $\upsilon
+ad_{\eta }^{\ast }\mu $ in $T_{\mu }\mathfrak{g}^{\ast }\simeq \mathfrak{g}%
^{\ast }$. Equating this to $\left\langle \chi _{1},X_{\left( \eta
,\upsilon ,\zeta \right) }^{\mathfrak{z}_{d}}\right\rangle $ in
Eq.(\ref{chi1}) defines the Lagrangian submanifold
$s_{\mathfrak{z}_{d}}^{\prime }=im\left( \mathfrak{d}h\right) $ in
Eq.(\ref{szd2}) given in coordinates $\left( \lambda ,\mu ,\xi
\right) $ of $\mathfrak{z}_{d}$ by
\begin{equation*}
\lambda =ad_{\frac{\delta h}{\delta \mu }}^{\ast }\mu \text{ \ \ and \ }\xi =%
\frac{\delta h}{\delta \mu }.
\end{equation*}
\end{proof}
\end{proposition}

\subsection{Legendre Transformation for Reduced Dynamics}

Being cotangent bundles, $T^{\ast }\mathfrak{g}=\mathfrak{\mathfrak{g}}%
\times \mathfrak{\mathfrak{g}^{\ast }}$ and $T^{\ast }\mathfrak{g}^{\ast }=%
\mathfrak{\mathfrak{g}}^{\ast }\times \mathfrak{\mathfrak{g}}$ are
canonically symplectic. It is possible to embed $T^{\ast }\mathfrak{g}$ and $%
T^{\ast }\mathfrak{g}^{\ast }$ symplectically into the total space $%
\mathfrak{z}_{d}$
\begin{eqnarray}
\hat{\varkappa} &:&T^{\ast }\mathfrak{g}\rightarrow
\mathfrak{z}_{d}:\left( \xi ,\mu \right) \rightarrow \left( ad_{\xi
}^{\ast }\mu ,\mu ,\xi \right) ,
\label{chis} \\
\hat{\omega} &:&T^{\ast }\mathfrak{g}^{\ast }\mathfrak{\rightarrow z}%
_{d}:\left( \mu ,\xi \right) \rightarrow \left( ad_{\xi }^{\ast }\mu
,\mu ,\xi \right) .  \label{ohmh}
\end{eqnarray}%
The following proposition shows how to define Lagrangian submanifold $%
im\left( \mathfrak{d}l\right) =s_{\mathfrak{z}_{d}}$ in
Eq.(\ref{szd}) from the right wing (that is from the Hamiltonian
side) of the reduced Tulczyjew's triplet (\ref{RTT}).

\begin{proposition}
The Lagrangian dynamics determined by the Lagrangian submanifold $s_{%
\mathfrak{z}_{d}}$ in Eq.(\ref{szd}) is generated by the Morse
family
\begin{equation}
E^{l\rightarrow h}=\left( l\circ \tau _{\mathfrak{z}_{d}}\right)
+\Delta =l\left( \xi \right) -\left\langle \mu ,\xi \right\rangle
\label{RedLeg}
\end{equation}%
on the bundle $\mathfrak{\mathfrak{g}}\times \mathfrak{\mathfrak{g}^{\ast }}%
\rightarrow \mathfrak{\mathfrak{g}^{\ast }}$. Here, $\Delta $ is a
real
valued function on $\mathfrak{\mathfrak{g}}\times \mathfrak{\mathfrak{g}%
^{\ast }}$ obtained from the equation%
\begin{equation*}
\chi _{2}-\chi _{1}=d\Delta =d\left\langle \mu ,\xi \right\rangle
\end{equation*}%
where $\chi _{1}$ and $\chi _{2}$ are given in Eqs.(\ref{chi1}) and (\ref%
{chi2}), respectively.

\begin{proof}
In the trivialized cases, the Legendre transformations have been
achieved by Morse families on the trivialized Pontryagin bundle
$G\circledS \left( \mathfrak{\mathfrak{g}}\times
\mathfrak{\mathfrak{g}^{\ast }}\right) $. For the reduced dynamics,
due to the invariance under the group action $G$, the Morse families
will be defined on $G\backslash \left( G\circledS \left(
\mathfrak{\mathfrak{g}}\times \mathfrak{\mathfrak{g}^{\ast }}\right)
\right) \simeq \mathfrak{\mathfrak{g}}\times
\mathfrak{\mathfrak{g}^{\ast }}$.
Recall the generating object%
\begin{equation*}
\xymatrix{\mathfrak{z}_{d} \ar[drr]^{\pi
_{\mathfrak{z}_{d}}}&&\mathfrak{g}^{\ast
}\times\mathfrak{g}\ar[d]\ar[ll]_{\hat{\omega}}\\&&\mathfrak{g}^{\ast
}}
\end{equation*}
where $\hat{\omega}$ is the embedding in Eq.(\ref{ohmh}). According
to general theory of generating families (c.f. Eq.(\ref{LagSub})),
the Morse
family $E^{l\rightarrow h}$ generates a Lagrangian submanifold $s_{T^{\ast }%
\mathfrak{g}^{\ast }}$ of $T^{\ast }\mathfrak{g}^{\ast }$ given by%
\begin{equation}
s_{T^{\ast }\mathfrak{g}^{\ast }}=\left\{ \left( \mu ,\xi \right)
\in T^{\ast }\mathfrak{g}^{\ast }:T^{\ast }\pi
_{\mathfrak{z}_{d}}\left( \mu ,\xi \right) =dE^{l\rightarrow
h}\left( \mu ,\xi \right) \right\} .
\end{equation}%
Explicitly, the Lagrangian submanifold $s_{T^{\ast
}\mathfrak{g}^{\ast }}$ consists of two-tuples $\left( \delta
l/\delta \xi ,\xi \right) $. Hence, the image of $s_{T^{\ast
}\mathfrak{g}^{\ast }}$ under the map $\hat{\omega}$ is
$s_{\mathfrak{z}_{d}}$. When the Lagrangian $l$ is not regular then
it is
not possible to define a function $h$ on $\mathfrak{g}^{\ast }$ generating $%
s_{\mathfrak{z}_{d}}$. In this case, we only have Morse family $%
E^{l\rightarrow h}$.
\end{proof}
\end{proposition}

\begin{remark}
Recall that, the Morse family $E^{\bar{L}\rightarrow \bar{H}}$,
defined in Eq.(\ref{Morse1}), is also a generating family for
Euler-Poincar\'{e} equations. Here, the function
$E^{\bar{L}\rightarrow \bar{H}}$ is defined on
the Pontryagin bundle $^{1}PG$ with base manifold $G\circledS \mathfrak{g}%
^{\ast }$ whereas $E^{l\rightarrow h}$, in Eq.(\ref{RedLeg}), is defined on $%
T^{\ast }\mathfrak{\mathfrak{g}}^{\ast }\mathfrak{\ }$with cotangent
bundle projection.
\end{remark}

Now, we will establish the inverse Legendre transformation. The
Hamiltonian dynamics is defined by a Hamiltonian function $h$ on
$\mathfrak{g}^{\ast }$. A Hamiltonian functional on
$\mathfrak{g}^{\ast }$ defines the Lagrangian submanifold
$s_{\mathfrak{z}_{d}}^{\prime }$, in Eq.(\ref{szd2}), of $\left(
\mathfrak{z}_{d},\Omega _{\mathfrak{z}_{d}}\right) $. The following
proposition shows how to generate $s_{\mathfrak{z}_{d}}^{\prime }$
using the
left wing (that is the Lagrangian side) of the reduced Tulczyjew's triplet (%
\ref{RTT}).

\begin{proposition}
The Morse family
\begin{equation}
E^{h\rightarrow l}=\Delta -h\left( \mu \right) =\left\langle \mu
,\xi \right\rangle -h\left( \mu \right)  \label{Ehl}
\end{equation}%
on the bundle $\mathfrak{\mathfrak{g}}\times \mathfrak{\mathfrak{g}^{\ast }}%
\rightarrow \mathfrak{\mathfrak{g}}$ generates the Lagrangian submanifold $%
s_{\mathfrak{z}_{d}}^{\prime }$ in Eq.(\ref{szd2}).

\begin{proof}
In this case, diagram is%
\begin{equation*}
\xymatrix{\mathfrak{g}^{\ast
}\times\mathfrak{g}\ar[d]\ar[rr]^{\hat{\varkappa}}&&\mathfrak{z}_{d}\ar[dll]^{\tau_{\mathfrak{z}_{d}}}
\\\mathfrak{g}}
\end{equation*}
where $\hat{\varkappa}$ is the embedding, in Eq.(\ref{chis}), of $T^{\ast }%
\mathfrak{g}$ into $\mathfrak{z}_{d}$. The Morse family
$E^{h\rightarrow l}$ in Eq.(\ref{Ehl}) generates a Lagrangian
submanifold
\begin{equation}
s_{T^{\ast }\mathfrak{g}}=\left\{ \left( \xi ,\mu \right) \in T^{\ast }%
\mathfrak{g}:T^{\ast }\tau _{\mathfrak{z}_{d}}\left( \xi ,\mu
\right) =dE^{h\rightarrow l}\left( \xi ,\mu \right) \right\} ,
\end{equation}%
of $T^{\ast }\mathfrak{g}$ (c.f. Eq.(\ref{LagSub})). This Lagrangian
submanifold consists of two-tuples $\left( \delta h/\delta \mu ,\mu \right) $%
. $\hat{\varkappa}$ maps $s_{T^{\ast }\mathfrak{g}}$ to $s_{\mathfrak{z}%
_{d}}^{\prime }.$ When the Hamiltonian $h$ is not regular then it is
not possible to find a Lagrangian function $l$ on $\mathfrak{g}$. In
this case, we only have Morse family $E^{h\rightarrow l}$.
\end{proof}
\end{proposition}

\begin{remark}
Recall that, the Morse family $E^{\bar{H}\rightarrow \bar{L}}$,
defined in
Eq.(\ref{Leg2}), is another generating family for Lie-Poisson dynamics. $E^{%
\bar{H}\rightarrow \bar{L}}$ is defined on the Pontryagin bundle
$^{1}PG$ with base manifold $G\circledS \mathfrak{g}$ whereas
$E^{l\rightarrow h}$, in Eq.(\ref{Ehl}), is defined on $T^{\ast
}\mathfrak{\mathfrak{g}\ }$with tangent bundle projection.
\end{remark}

To summarize, in order to use the classical formulations of
generating
objects, we employ the following reduced form of Tulczyjew's triplet%
\begin{equation*}
\xymatrix{T^{\ast
}\mathfrak{g}\ar[dr]\ar[rr]^{\hat{\varkappa}}&&\mathfrak{z}_{d}\ar[dl]^{\tau_{\mathfrak{z}_{d}}}\ar[dr]_{\pi_{\mathfrak{z}_{d}}}&&T^{\ast
}\mathfrak{g}{\ast }\ar[ll]_{\hat{\omega}}\ar[dl]
\\&\mathfrak{g}&&\mathfrak{g}^{\ast}&}
\end{equation*}
where we replace the total spaces $\mathfrak{z}_{l}$ and
$\mathfrak{z}_{h}$ by $T^{\ast }\mathfrak{g}$ and $T^{\ast
}\mathfrak{g}^{\ast }$,
respectively. The payoff is that the mappings $\hat{\varkappa}$ and $\hat{%
\omega}$ are symplectic embeddings but not isomorphisms.

\section{Example:\ Diffeomorphism Groups}

\subsection{Group Structure}

Group $\mathcal{D}$ of diffeomorphisms on $\mathcal{Q}$ is a Lie group (see for example \cite%
{Ba97, EsGu11, RaSc81}). Lie algebra of $\mathcal{D}$ is the space $\mathfrak{%
X}$ of vector fields on $\mathcal{Q}$. The (right) adjoint action $Ad$ of $%
\mathcal{D}$ on $\mathfrak{X}$ is given by the pull-back operation
$\varphi ^{\ast }X$, for $\varphi \in G$ and $X\in \mathfrak{X}$.
The infinitesimal adjoint action of an element $Y\in \mathfrak{X}$
on $X\in \mathfrak{X}$ is
the Lie derivative of $X$ in the direction of $Y$, that is $\mathcal{L}_{Y}X$%
. The tangent space
\begin{equation*}
T_{\varphi }\mathcal{D}=\left\{ X_{\varphi }:\mathcal{Q}\rightarrow T%
\mathcal{Q}:X_{\varphi }=X\circ \varphi \text{ for some }X\in \mathfrak{X}%
\right\}
\end{equation*}%
at $\varphi \in \mathcal{D}$ consists of material velocity fields.
The lifted group multiplication on the tangent bundle $T\mathcal{D}$
is
\begin{equation}
\varpi _{T\mathcal{D}}\left( X_{\varphi },Y_{\psi }\right)
=X_{\varphi \circ \psi }+T\varphi \circ Y_{\psi }.  \label{GrTGsu}
\end{equation}%
The right and the left trivializations of $T\mathcal{D}$ are
\begin{eqnarray}
tr_{T\mathcal{D}}^{R} &:&T\mathcal{D}\rightarrow \mathcal{D}\times \mathfrak{%
X}:X_{\varphi }\rightarrow \left( \varphi ,X\right)  \label{trRsu} \\
tr_{T\mathcal{D}}^{L} &:&T\mathcal{D}\rightarrow \mathcal{D}\times \mathfrak{%
X}:X_{\varphi }\rightarrow \left( \varphi ,\varphi ^{\ast }X\right)
.  \notag
\end{eqnarray}%
After choosing the right trivialization $tr_{T\mathcal{D}}^{R}$, we
arrive at the semidirect product group multiplication
\begin{equation}
\varpi _{\mathcal{D}\circledS \mathfrak{X}}\left( \left( \varphi
,X\right) ,\left( \psi ,Y\right) \right) =\left( \varphi \psi
,X+\varphi _{\ast }Y\right)  \label{GrTGSDsu}
\end{equation}%
on $\mathcal{D}\circledS \mathfrak{X}$. The Lie algebra of $\mathcal{D}%
\circledS \mathfrak{X}$ is $\mathfrak{X}\circledS \mathfrak{X}$ with
semi-direct product%
\begin{equation*}
\lbrack \left( X_{1},X_{2}\right) ,\left( Y_{1},Y_{2}\right) ]_{\mathfrak{X}%
\circledS \mathfrak{X}}=\left( \left[ X_{1},Y_{1}\right] ,\left[ X_{1},Y_{2}%
\right] -\left[ Y_{1},X_{2},\right] \right) .
\end{equation*}

The dual space $\mathfrak{X}^{\ast }\ $of the Lie algebra
$\mathfrak{X}$ is
the space $\Lambda ^{1}\left( \mathcal{Q}\right) \otimes Den\left( \mathcal{Q%
}\right) $ of one-form densities on $\mathcal{Q}$. The pairing
between $\mu \otimes d^{n}q\in \mathfrak{X}^{\ast }$ and $X\in
\mathfrak{X}$ is given by
the integration%
\begin{equation}
\left\langle \mu \otimes d^{n}q,X\right\rangle =\int_{\mathcal{M}%
}\left\langle \mu ,X\right\rangle _{\mathcal{Q}}d^{n}q,
\end{equation}%
where $d^{n}q$ is the top-form on $\mathcal{Q}$. The pairing inside
the
integral is the natural pairing of finite dimensional spaces $T_{q}\mathcal{Q%
}$ and $T_{q}^{\ast }\mathcal{Q}$. The coadjoint action $Ad^{\ast }$ of $%
\mathcal{D}$ on $\mathfrak{X}^{\ast }$ is the pull-back operation
$\varphi ^{\ast }\left( \mu \otimes d^{n}q\right) $ for $\varphi \in
\mathcal{D}$ and $\mu \otimes d^{n}q\in \mathfrak{X}^{\ast }$. The
infinitesimal coadjoint action $ad^{\ast }$ of an element $X\in
\mathfrak{X}$ on $\mu \otimes d^{n}q\in \mathfrak{X}^{\ast }$ is
minus the Lie derivative of $\mu \otimes
d^{n}q$ by $X$, that is%
\begin{equation}
ad_{X}^{\ast }:\mathfrak{X}^{\ast }\rightarrow \mathfrak{X}^{\ast
}:\mu \otimes d^{n}q\rightarrow -\left( \mathcal{L}_{X}\mu +\left(
div_{d^{n}q}X\right) \mu \right) \otimes d^{n}q.
\end{equation}%
Here, $div_{d^{n}q}X$ denotes divergence of the vector field $X$
with respect to the top-form $d^{n}q$. The cotangent space at
$\varphi $ is
\begin{equation*}
T_{\varphi }^{\ast }\mathcal{D}=\left\{ \left( \mu _{\varphi }:\mathcal{Q}%
\rightarrow T^{\ast }\mathcal{Q}\right) \otimes d^{n}q:\mu _{\varphi
}=\mu \circ \varphi \text{, }\mu \in \Lambda ^{1}\left(
\mathcal{Q}\right) \right\} .
\end{equation*}%
The pairing between $T_{\varphi }^{\ast }\mathcal{D}$ and $T_{\varphi }%
\mathcal{D}$ is taken to be the right invariant $L^{2}-$integral.
Cotangent lifts of right and left actions of $\mathcal{D}$ on
$T^{\ast }\mathcal{D}$ can be computed using
\begin{equation*}
T_{\varphi \circ \psi }^{\ast }R_{\psi ^{-1}}\left( \mu _{\varphi
}\right) =\mu _{\varphi \circ \psi }\text{, \ \ }T_{\psi \circ
\varphi }^{\ast
}L_{\psi ^{-1}}\mu _{\varphi }=T^{\ast }\psi ^{-1}\circ \mu _{\varphi }\text{%
,}
\end{equation*}%
respectively. The cotangent bundle $T^{\ast }\mathcal{D}$ is a Lie
group
with the group multiplication%
\begin{equation*}
\left( \mu _{\varphi },\nu _{\psi }\right) =\mu _{\varphi \circ \psi
}+T^{\ast }\varphi ^{-1}\circ \nu _{\psi }.
\end{equation*}%
The right and left trivializations of $T^{\ast }\mathcal{D}$ are
\begin{eqnarray}
tr_{T^{\ast }\mathcal{D}}^{R} &:&T^{\ast }\mathcal{D}\rightarrow \mathcal{D}%
\times \mathfrak{X}^{\ast }:\mu _{\varphi }\otimes d^{n}q\rightarrow
\left(
\varphi ,\mu \otimes d^{n}q\right)  \notag  \label{trT*} \\
tr_{T^{\ast }\mathcal{D}}^{L} &:&T^{\ast }\mathcal{D}\rightarrow \mathcal{D}%
\times \mathfrak{X}^{\ast }:\mu _{\varphi }\otimes d^{n}q\rightarrow
\left( \varphi ,\varphi ^{\ast }\mu \otimes \varphi ^{\ast
}d^{n}q\right) .
\end{eqnarray}%
We choose the right trivialization to arrive at the semi-direct
product group structure with multiplication
\begin{equation*}
\varpi _{\mathcal{D}\circledS \mathfrak{X}^{\ast }}\left( \left(
\varphi ,\mu \otimes d^{n}q\right) ,\left( \psi ,\nu \otimes
d^{n}q\right) \right) =\left( \varphi \circ \psi ,\left( \mu
+\varphi _{\ast }\nu \right) \otimes \varphi _{\ast }d^{n}q\right)
\end{equation*}%
on the trivialization $\mathcal{D}\circledS \mathfrak{X}^{\ast }$.

\subsection{The Trivialized Dynamics}

Let $\bar{L}=\bar{L}\left( \varphi ,X\right) $ be a Lagrangian density on $%
\mathcal{D}\circledS \mathfrak{X}$, the trivialized Euler-Lagrange
equations are
\begin{equation}
\frac{d}{dt}\frac{\delta \bar{L}}{\delta X}=\frac{\delta
\bar{L}}{\delta
\varphi }\circ \varphi ^{-1}-\mathcal{L}_{X}\frac{\delta \bar{L}}{\delta X}%
-\left( div_{d^{n}q}X\right) \frac{\delta \bar{L}}{\delta X}.
\label{tr-EL-Diff}
\end{equation}%
$\bar{L}$ generates a Lagrangian submanifold%
\begin{equation}
S_{^{1}TT^{\ast }\mathcal{D}}=\left( \varphi ,\frac{\delta \bar{L}}{\delta X}%
,X,\frac{\delta \bar{L}}{\delta \varphi }\circ \varphi ^{-1}\right)
\end{equation}%
of the trivialized Tulczyjew's symplectic space $^{1}TT^{\ast
}\mathcal{D}$
defined by the semi-direct product $\left( \mathcal{D}\circledS \mathfrak{X}%
^{\ast }\right) \circledS \left( \mathfrak{X}\circledS
\mathfrak{X}^{\ast }\right) $. Here, the trivialization map
\begin{equation}
T\left( \mathcal{D}\circledS \mathfrak{X}^{\ast }\right) \rightarrow \text{ }%
^{1}TT^{\ast }\mathcal{D}:\left( X_{\varphi },Y_{\mu }\right)
\rightarrow \left( \varphi ,\mu \otimes d^{n}q,X,Y_{\mu
}+\mathcal{L}_{X}\mu +\left( div_{d^{n}q}X\right) \mu \otimes
d^{n}q\right)   \label{tr-diff}
\end{equation}%
realizes the relation between the Lagrangian submanifold $S_{^{1}TT^{\ast }%
\mathcal{D}}$ and the trivialized Euler-Lagrange equation (\ref{tr-EL-Diff}%
). The Legendre transformation of trivialized Euler-Lagrange
equation can be achieved by the Morse family
\begin{equation*}
E\left( \varphi ,\mu \otimes d^{n}q,X\right) =\bar{L}\left( \varphi
,X\right) -\int_{\mathcal{Q}}\left\langle \mu ,X\right\rangle _{\mathcal{Q}%
}d^{n}q
\end{equation*}%
on the Pontryagin bundle $\mathcal{D}\circledS \left(
\mathfrak{X}^{\ast }\oplus \mathfrak{X}\right) $ over
$\mathcal{D}\circledS \mathfrak{X}^{\ast } $.

A right invariant vector field on $\mathcal{D}\circledS
\mathfrak{X}^{\ast }$
is given by%
\begin{equation*}
X_{\left( X,\nu \right) }^{\mathcal{D}\circledS \mathfrak{X}^{\ast
}}\left( \varphi ,\mu \right) =\left( X_{\varphi },\nu
-\mathcal{L}_{X}\mu -\left( div_{d^{n}q}X\right) \mu \otimes
d^{n}q\right) .
\end{equation*}%
At ($\varphi ,\mu \otimes d^{n}q$), the values of canonical one-form
$\theta
_{\mathcal{D}\circledS \mathfrak{X}^{\ast }}$\ and the symplectic two-form $%
\Omega _{\mathcal{D}\circledS \mathfrak{X}^{\ast }}$ on the right
invariant vector fields are
\begin{eqnarray*}
\left\langle \theta _{\mathcal{D}\circledS \mathfrak{X}^{\ast
}},X_{\left( X,\nu \right) }^{\mathcal{D}\circledS
\mathfrak{X}^{\ast }}\right\rangle &=&\int_{\mathcal{Q}}\left\langle
\mu ,X\right\rangle _{\mathcal{Q}}d^{n}q,
\\
\left\langle \Omega _{\mathcal{D}\circledS \mathfrak{X}^{\ast
}};\left( X_{\left( X,\nu \right) }^{\mathcal{D}\circledS
\mathfrak{X}^{\ast }},X_{\left( Y,\lambda \right)
}^{\mathcal{D}\circledS \mathfrak{X}^{\ast }}\right) \right\rangle
&=&\int_{\mathcal{Q}}\left\langle \nu
,Y\right\rangle _{\mathcal{Q}}-\left\langle \lambda ,X\right\rangle _{%
\mathcal{Q}}+\left\langle \mu ,\left[ X,Y\right] \right\rangle _{\mathcal{Q}%
}d^{n}q.
\end{eqnarray*}%
For a Hamiltonian function $\bar{H}$ on $\mathcal{D}\circledS \mathfrak{X}%
^{\ast }$, the trivialized Hamilton's equations are
\begin{equation}
\frac{d\varphi }{dt}=\left( \frac{\delta \bar{H}}{\delta \mu
}\right)
_{\varphi },\ \text{\ }\ \frac{d\mu }{dt}=-\mathcal{L}_{\frac{\delta \bar{H}%
}{\delta \mu }}\mu -\left( div_{d^{n}q}\frac{\delta \bar{H}}{\delta \mu }%
\right) \mu -\left( \frac{\delta \bar{H}}{\delta \varphi }\right)
\circ \varphi ^{-1},  \label{tr-Ham-Diff}
\end{equation}%
where, due to the reflexivity assumption, $\delta \bar{H}/\delta \mu
$ is
assumed to be a vector field in the Lie algebra, and $\left( \delta \bar{H}%
/\delta \mu \right) _{\varphi }$ is the material velocity field. The
Lagrangian submanifold generated by the Hamiltonian function
$\bar{H}$ is
\begin{equation*}
S_{^{1}TT^{\ast }\mathcal{D}}^{\prime }=\left( \varphi ,\mu \otimes d^{n}q,%
\frac{\delta \bar{H}}{\delta \mu },-\frac{\delta \bar{H}}{\delta \varphi }%
\circ \varphi ^{-1}\otimes d^{n}q\right) .
\end{equation*}%
To establish the link between $S_{^{1}TT^{\ast }\mathcal{D}}^{\prime
}$ and Eq.(\ref{tr-Ham-Diff}), we refer to the trivialization map
(\ref{tr-diff}).
The Legendre transformation of the Hamiltonian dynamics described by Eqs.(%
\ref{tr-Ham-Diff}) results from the Morse family
\begin{equation*}
E\left( \varphi ,\mu \otimes d^{n}q,X\right)
=\int_{\mathcal{Q}}\left\langle \mu ,X\right\rangle
_{\mathcal{Q}}d^{n}q-\bar{H}\left( \varphi ,\mu \right)
\end{equation*}%
on the Pontryagin bundle $\mathcal{D}\circledS \left(
\mathfrak{X}^{\ast }\oplus \mathfrak{X}\right) $ over
$\mathcal{D}\circledS \mathfrak{X}$.

\subsection{The Reduced Dynamics}

When the Lagrangian $\bar{L}$ is free of the group variable, we have $\bar{L}%
=l\left( X\right) $ and the trivialized Euler-Lagrange equations (\ref%
{tr-EL-Diff}) reduces to the Euler-Poincar\'{e} equations
\begin{equation}
\frac{d}{dt}\frac{\delta l}{\delta X}=-\mathcal{L}_{X}\frac{\delta
l}{\delta X}-\left( div_{d^{n}q}X\right) \frac{\delta l}{\delta X}.
\end{equation}%
Similarly, when the Hamiltonian depends on the fiber variable $\mu $ only, $%
\bar{H}\left( g,\mu \right) =h\left( \mu \right) $, the trivialized
Hamilton's equation (\ref{tr-Ham-Diff}) gives the Lie-Poisson equation%
\begin{equation}
\frac{d\mu }{dt}=-\mathcal{L}_{\frac{\delta h}{\delta \mu }}\mu
-\left( div_{d^{n}q}\frac{\delta h}{\delta \mu }\right) \mu .
\label{LP-Diff}
\end{equation}

In order to perform the Legendre transformations of the reduced
dynamics, we present them as Lagrangian submanifolds of the reduced
Tulczyjew's symplectic space $\mathcal{O}_{\lambda }\times
\mathfrak{X}^{\ast }\times \mathfrak{X}$. Here,
$\mathcal{O}_{\lambda }\times \mathfrak{X}^{\ast }\times
\mathfrak{X}$ is obtained by application of the Marsden-Weinstein
reduction to the trivialized Tulczyjew's symplectic space $^{1}TT^{\ast }%
\mathcal{D}$, that is

\begin{equation}
^{1}TT^{\ast }\mathcal{D}\rightarrow \mathcal{O}_{\lambda }\times \mathfrak{X%
}^{\ast }\times \mathfrak{X}:\left( \varphi ,\mu \otimes
d^{n}q,X,\nu \otimes d^{n}q\right) \rightarrow \left( \varphi _{\ast
}\left( \lambda \otimes d^{n}q\right) ,\mu \otimes d^{n}q,X\right) ,
\label{sr-Diff}
\end{equation}%
where $\lambda =\nu -\mathcal{L}_{X}\mu -\left( div_{d^{n}q}X\right)
\mu $. For a Lagrangian $l$ on $\mathfrak{X}$, image of the
Lagrange-Dirac derivative
\begin{equation*}
\mathfrak{d}l:\mathfrak{X}\rightarrow \mathcal{O}_{\lambda }\times \mathfrak{%
X}^{\ast }\times \mathfrak{X}:X\rightarrow \left( -\mathcal{L}_{X}\frac{%
\delta l}{\delta X}-\left( div_{d^{n}q}X\right) \frac{\delta l}{\delta X}%
\otimes d^{n}q,\frac{\delta l}{\delta X}\otimes d^{n}q,X\right)
\end{equation*}%
is a Lagrangian submanifold of $\mathcal{O}_{\lambda }\times \mathfrak{X}%
^{\ast }\times \mathfrak{X}$. The image $im\left(
\mathfrak{d}l\right) $ determines Euler-Poincar\'{e} equations. The
Legendre transformation is generated by the Morse family
\begin{equation*}
E^{l\rightarrow h}\left( \mu \otimes d^{n}q,X\right) =l\left(
X\right) -\int_{\mathcal{Q}}\left\langle \mu ,X\right\rangle
_{\mathcal{Q}}d^{n}q
\end{equation*}%
on the bundle $\mathfrak{X}^{\ast }\times \mathfrak{X}\rightarrow \mathfrak{X%
}^{\ast }$. Similarly, for a Hamiltonian $h$ on $\mathfrak{X}^{\ast
}$, the
image of Hamilton-Dirac derivative%
\begin{equation*}
-\mathfrak{d}h:\mathfrak{X}^{\ast }\rightarrow \mathcal{O}_{\lambda
}\times
\mathfrak{X}^{\ast }\times \mathfrak{X}:\mu \rightarrow \left( -\mathcal{L}_{%
\frac{\delta h}{\delta \mu }}\mu -\left( div_{d^{n}q}\frac{\delta
h}{\delta \mu }\right) \mu \otimes d^{n}q,\mu \otimes
d^{n}q,\frac{\delta h}{\delta \mu }\right)
\end{equation*}%
is a Lagrangian submanifold of $\mathcal{O}_{\lambda }\times \mathfrak{X}%
^{\ast }\times \mathfrak{X}$ and determines Lie-Poisson equations.
The inverse Legendre transformation of the Euler-Poincar\'{e}
dynamics is generated by the Morse family
\begin{equation}
E^{h\rightarrow l}\left( \mu \otimes d^{n}q,X\right) =\int_{\mathcal{Q}%
}\left\langle \mu ,X\right\rangle _{\mathcal{Q}}d^{n}q-h\left( \mu
\right)
\end{equation}%
on the bundle $\mathfrak{X}^{\ast }\times \mathfrak{X}\rightarrow \mathfrak{X%
}$.

\section{Summary, Discussions and Prospectives}

We obtain trivialized and reduced dynamics for Hamiltonian and
Lagrangian formulations of systems with configuration space $G$.
Following diagram summarizes these equations and their
representations by Lagrangian
submanifolds.%
\begin{equation*}
\begin{tabular}{|l|l|l|}
\hline & Dynamics & Lagrangian Submanifold \\ \hline
\begin{tabular}{l}
Trivialized \\
Euler-Lagrange \\
Equations \\
on $G\circledS \mathfrak{g}$%
\end{tabular}
& $\frac{d}{dt}\frac{\delta \bar{L}}{\delta \xi }=T_{e}^{\ast }R_{g}\frac{%
\delta \bar{L}}{\delta g}+ad_{\xi }^{\ast }\frac{\delta
\bar{L}}{\delta \xi }
$ & $\left\{ g,\frac{\delta \bar{L}}{\delta \xi },\xi ,T^{\ast }R_{g}\frac{%
\delta \bar{L}}{\delta g}\right\} \subset $ $^{1}TT^{\ast }G$ \\
\hline
\begin{tabular}{l}
Trivialized \\
Hamilton's \\
Equations \\
on $G\circledS \mathfrak{g}^{\ast }$%
\end{tabular}
&
\begin{tabular}{l}
$\frac{dg}{dt}=T_{e}R_{g}\left( \frac{\delta \bar{H}}{\delta \mu
}\right) ,$
\\
$\frac{d\mu }{dt}=ad_{\frac{\delta \bar{H}}{\delta \mu }}^{\ast }\mu
-T_{e}^{\ast }R_{g}\frac{\delta \bar{H}}{\delta g}$%
\end{tabular}
& $\left\{ g,\mu ,\frac{\delta \bar{H}}{\delta \mu },-T^{\ast }R_{g}\frac{%
\delta \bar{H}}{\delta g}\right\} \subset $ $^{1}TT^{\ast }G$ \\
\hline
\begin{tabular}{l}
Euler-Poincar\'{e} \\
Equations \\
on $\mathfrak{g}$%
\end{tabular}
& $ad_{\xi }^{\ast }\frac{\delta l}{\delta \xi }-\frac{d}{dt}\frac{\delta l}{%
\delta \xi }=0$ & $\left\{ ad_{\xi }^{\ast }\frac{\delta l}{\delta \xi },%
\frac{\delta l}{\delta \xi },\xi \right\} \subset \mathfrak{z}_{d}$
\\ \hline
\begin{tabular}{l}
Lie-Poisson \\
Equations \\
on $\mathfrak{g}^{\ast }$%
\end{tabular}
& $\frac{d\mu }{dt}=ad_{\frac{\delta h}{\delta \mu }}^{\ast }\mu $ & $%
\left\{ ad_{\frac{\delta h}{\delta \mu }}^{\ast }\mu ,\mu ,\frac{\delta h}{%
\delta \mu }\right\} \subset \mathfrak{z}_{d}$ \\ \hline
\end{tabular}%
\end{equation*}
We identify the following
Morse families for trivialized and reduced dynamics.
\begin{equation*}
\begin{tabular}{|l|c|c|}
\hline & Morse family & Bundle \\ \hline
\begin{tabular}{l}
Trivialized \\
Euler-Lagrange \\
Equations \\
on $G\circledS \mathfrak{g}$%
\end{tabular}
& $E^{\bar{L}\rightarrow \bar{H}}\left( g,\xi ,\mu \right)
=\bar{L}\left( g,\xi \right) -\left\langle \mu ,\xi \right\rangle $
& $G\circledS \left( \mathfrak{g}\times \mathfrak{g}^{\ast }\right)
\rightarrow G\circledS \mathfrak{g}^{\ast }$ \\ \hline
\begin{tabular}{l}
Trivialized \\
Hamilton's \\
Equations \\
on $G\circledS \mathfrak{g}^{\ast }$%
\end{tabular}
& $E^{\bar{H}\rightarrow \bar{L}}\left( g,\xi ,\mu \right)
=\left\langle \mu ,\xi \right\rangle -\bar{H}\left( g,\mu \right) $
& $G\circledS \left( \mathfrak{g}\times \mathfrak{g}^{\ast }\right)
\rightarrow G\circledS \mathfrak{g}$ \\ \hline
\begin{tabular}{l}
Euler-Poincar\'{e} \\
Equations \\
on $\mathfrak{g}$%
\end{tabular}
& $E^{l\rightarrow h}\left( \xi ,\mu \right) =l\left( \xi \right)
-\left\langle \mu ,\xi \right\rangle $ & $\mathfrak{g}\times \mathfrak{g}%
^{\ast }\rightarrow \mathfrak{g}^{\ast }$ \\ \hline
\begin{tabular}{l}
Lie-Poisson \\
Equations \\
on $\mathfrak{g}^{\ast }$%
\end{tabular}
& $E^{h\rightarrow l}\left( \xi ,\mu \right) =\left\langle \mu ,\xi
\right\rangle -h\left( \mu \right) $ & $\mathfrak{g}\times \mathfrak{g}%
^{\ast }\rightarrow \mathfrak{g}$. \\ \hline
\end{tabular}%
\end{equation*}

Obviously, the form of dynamical equations obtained in this work
depends on the trivialization we employed. What we refer to
trivialization of the first kind carries the group operations to
iterated bundles and contributes additional term due to semi-direct
product structures. Higher order dynamics on Lie groups with adapted
trivializations of higher order and iterated bundles as well as
their symplectic and Poisson reductions are under investigation
\cite{Es14}.

The reduction of Tulczyjew's symplectic space can be generalized to
symplectic reduction of tangent bundle of a symplectic manifold with
the lifted symplectic structure. This could be the first step toward
more general studies on the reduction of the special symplectic
structures and the reduction of Tulczyjew's triplet with
configuration manifold $\mathcal{Q} $.

Finally, we want to mention that the foremost example of degenerate
system that falls into application area of present formulation is
the Vlasov-Poisson equation of plasma dynamics which was, indeed,
the motivation for this work.

\end{document}